\documentclass[leqno]{amsart}
\usepackage{amsfonts,amssymb,amsmath,amsgen,amsthm}
\usepackage{fancyhdr}

\theoremstyle{plain}
\newtheorem{theorem}{Theorem}[section]
\newtheorem{definition}[theorem]{Definition}
\newtheorem{assumption}[theorem]{Assumption}
\newtheorem{lemma}[theorem]{Lemma}
\newtheorem{corollary}[theorem]{Corollary}
\newtheorem{proposition}[theorem]{Proposition}

\theoremstyle{remark}
\newtheorem{remark}[theorem]{Remark}
\newtheorem*{notation}{Notation}
\newtheorem{example}[theorem]{Example}

\def\Tend#1#2{\mathop{\longrightarrow}\limits_{#1\rightarrow#2}}

\font\tenms=msbm10
\font\sevenms=msbm7
\font\fivems=msbm5
\newfam\bbfam \textfont\bbfam=\tenms \scriptfont\bbfam=\sevenms
\scriptscriptfont\bbfam=\fivems

\let\vf=\varphi
\let\om=\omega

\def\eps{\varepsilon}

\def\cdotv{\raise 2pt\hbox{,}}

\def\C{{\mathbf C}}
\def\R{{\mathbf R}}
\def\N{{\mathbf N}}
\def\Sch{{\mathcal S}}

\def\virgp{\raise 2pt\hbox{,}}

\def\({\left(}
\def\){\right)}
\def\<{\left\langle}
\def\>{\right\rangle}
\def\le{\leqslant}

\def\d{{\partial}}

\def\Tend#1#2{\mathop{\longrightarrow}\limits_{#1\rightarrow#2}}

\numberwithin{equation}{section}

\title{Coherent states for systems of $L^2-$supercritical nonlinear Schr\"odinger equations.}
\author[L. ~Hari]{Lysianne Hari}
\address[L. Hari]{University of Cergy-Pontoise\\
                  UMR CNRS 8088\\
                  F-95000 Cergy-Pontoise.}
\email{Lysianne.Hari@u-cergy.fr}

\begin{document}
\begin{abstract}
We consider the propagation of wave packets for a nonlinear Schr\"o\-dinger equation, with a matrix-valued potential, 
in the semi-classical limit. For a matrix-valued potential, Strichartz estimates are available under long range assumptions. 
Under these assumptions, for an initial coherent state polarized along an eigenvector, 
we prove that the wave function remains in the same eigenspace, in a scaling such that nonlinear effects cannot be 
neglected. 
We also prove a nonlinear superposition principle for these nonlinear wave packets.
\end{abstract}
\maketitle

\section{Introduction}
\label{Intro}
We consider the semi-classical limit $\eps\to 0$ for the nonlinear
Schr\"odinger equation 
\begin{equation}
\label{NLS0}
     i\eps\d_t \psi^\eps - P(\eps)\psi^\eps = \Lambda \eps ^{\beta}|\psi^\eps |_{\C^2}^{2} \psi^\eps
     \quad ; \quad
     \psi^\eps_{\mid t=0} = \psi^\eps_{0},
\end{equation}
where  $\Lambda \geq 0$,  $d\in \left\lbrace 2, 3 \right\rbrace$. The data $\psi^\eps_0$ and the solution
$\psi^\eps(t)$ are vectors of $\C^2$.  
The quantity $|\psi^\eps|_{\C^2}^{2}$ denotes the square of the
Hermitian norm in $\C^2$ of the vector~$\psi^\eps$, $P(\eps)$ is a matrix-valued Schr\"odinger operator acting on 
$L^2 (\R^d, \C^2)$,
$$ P(\eps) = - \dfrac{\eps^2}{2} \Delta \; \textrm{Id} \; + V(x),$$
where $V$ is a self-adjoint smooth $2 \times 2$ matrix depending on the parameter $x \in \R^d$ 
and the semiclassical parameter $\eps >0$ is small.
\\
\noindent
The data $\psi_0^{\eps}$ is a wave packet (or can be a perturbation or a sum of two wave packets) :
\begin{equation}
\label{data}
\psi^\eps_0(x)= \eps^{-d/4} e^{i\xi_0^+.
    (x-x_0^+)/\eps}a
\left(\frac{x-x_0^+}{\sqrt\eps}\right) \chi_+(x).
\end{equation}
The profile $a$ belongs to the Schwartz class, $a\in{\mathcal S}(\R^d)$,
and the initial datum is polarized along the eigenvector associated with $\lambda_+(x)$, 
$\chi_+(x)\in{\mathcal C} ^\infty(\R^d, \C^2)$:
$$V(x)\chi_+(x)=\lambda_+(x) \chi_+(x),\quad \text{ with }| \chi_+(x)|_{\C^2}=1.$$ 
We choose the critical exponent $\beta = \beta_c := 1+d/2$ : in the scalar case, the approximation of $\psi^\eps$ is a wave packet 
whose envelope satisfies a nonlinear equation, nonlinear effects can not be neglected (see \cite{CF-p}).
Moreover, let us notice that, contrary to the case $d=1$, the nonlinearity is not $L^2$-subcritical when $d=2$ or $3$, 
it is only $H^1$-subcritical. 
So the condition $\Lambda \geq 0$ is crucial here to avoid finite time blow-up (see \cite{Ca-p} and \cite{Glassey}).
\\
\noindent
The aim of the paper is to prove that the solutions of \eqref{NLS0} with initial data which are of the form \eqref{data} 
keep the same form and remain in the same eigenspace, when the potential satisfies assumptions that we are now going to explain. 
Note that the scalar case is studied in \cite{CF-p} 
and that matrix case, with $d=1$ is analysed by the authors of \cite{Adiab}.
\\
\noindent
We write $V$ as
$$ V(x) = \rho_0 (x) \; \textrm{Id} \; + \begin{pmatrix}
                                      \rho(x) & \om(x)\\
		
		      \om(x) & - \rho (x)
                                      \end{pmatrix}
$$
where the functions $\rho_0, \rho$ and $\om$ are smooth, and we make the following assumptions on $V$ :
\begin{assumption}
 \label{assumption}
$(i) \; V$ is long range : there exist a matrix $V_\infty$ and $p \in \R^+ \setminus \left\lbrace 0\right\rbrace$  
such that for $x \in \R^d$,
\begin{align*}
\exists C >0 ,\; \| V(x)-V_\infty \| &\leq C \left\langle x \right\rangle^{-p},\\
\forall \alpha \in \N^d, \; |\alpha|\geq 1, \; \exists C_\alpha >0, \;
 \| \d_x^\alpha V(x) \| &\leq C_\alpha \left\langle x \right\rangle^{-p- |\alpha|},
\end{align*}
where $\left\langle x \right\rangle = (1 + |x|^2)^{1/2}$ and the norm $\|.\|$ denotes the operator norm on $\C^{2,2}$.
\\
The eigenvalues of $V$ are given by :
\begin{equation}
 \label{eigenval}
\lambda_\pm (x) = \rho_0 (x) \pm \sqrt{\rho^2(x) + \om^2(x)}.
\end{equation}
We assume :
\\
$(ii) \; \exists \; \delta_0 >0, \; \; \om^2(x) + \rho^2(x) > \delta_0, \; \; \forall x \in \R^d.$ 
\\
This point guarantees that there exist smooth and normalized eigenvectors, $\chi_\pm(x)$ associated to $\lambda_\pm(x)$.
\\
$(iii) \; \exists K \subset \R^d$, $K$ a compact subset of $\R^d$ such that
$$ V(x) = \begin{pmatrix}
           \lambda_+(x) & 0 \\
            0 & \lambda_-(x)
          \end{pmatrix},
$$
for all $x \notin K$.
\end{assumption}
\begin{example}
\label{ex}
Potential satisfying Assumption \ref{assumption} can easily be found. 
For example, if we consider $\theta \in C_0^\infty(\R^d)$ and the potential
$$ V(x) = \left\langle x\right\rangle^{-p} \begin{pmatrix}
            \cos \; \theta(x) & \sin \; \theta(x) \\
            \sin \; \theta(x) & -\cos \; \theta(x)
          \end{pmatrix},$$
with $p>0$. Then $V$ satisfies all points of Assumption \ref{assumption}.
\end{example}
\begin{remark}
It is important to notice that $(i)$ and $(ii)$ of Assumption \ref{assumption} ensure that $\rho_0$, $\rho$ and $\om$ are bounded
with bounded derivatives and that the eigenvalues $\lambda_+(x)$ and $ \lambda_-(x)$ do not cross each other, 
which allows us to guarantee global smooth eigenvalues and eigenprojectors, satisfying
\begin{equation}
 \label{eigen}
\forall \alpha, \beta \in \N^d, \exists C >0, \forall x \in \R^d, |\d_x^\alpha \lambda_\pm (x)| + |\d_x^\beta \chi_\pm (x)| \leq C.
\end{equation} 
It is essential to obtain the main result of the paper : 
for example, for $d=2$, in the linear case, $\Lambda = 0$, if we consider the potential:
$$V(x) = \left( \begin{array}{cc}
  x_1 & x_2 \\
  x_2 & - x_1
 \end{array}\right), 
 $$
the eigenvalues $\lambda_\pm (x) = \pm |x|$ cross each other when $x = 0$. The authors of \cite{Hag94} and \cite{HJ99}
have proved that the approximation is not valid in this case and that there are exchanges of energies between different modes (see also 
\cite{modes} and \cite{Wigner}).
\end{remark}
\begin{remark}
Points $(i)$ and $(iii)$ are technical assumptions : under point $(i)$ of Assumption \ref{assumption}, for fixed $\eps >0$, 
we can prove global existence of the solution $\psi^\eps$. 
Actually, global existence can be proved under weaker conditions on the potential, this result is discussed in the appendix.
\\
\noindent
Moreover, they are useful to obtain Strichartz estimates, which will be crucial tools 
in the analysis. 
Thanks to point $(iii)$, we have constant eigenprojectors, outside a compact subset of $\R^d$, which is needed in \cite{FerRou} to obtain 
Strichartz estimates without any localization (for a deeper discussion about them, we refer to Section~\ref{estimates}).
\end{remark}
\noindent
We introduce the following notation :
\begin{notation}
 For two positive numbers $a^\eps$ and $b^\eps$, the notation $a^\eps \lesssim b^\eps$ means that there exists a constant $C>0$, 
independent of $\eps$, such that for all $\eps \in ]0,1]$, $a^\eps \leq C b^\eps$.
\end{notation}
\subsection{Classical trajectories}
\noindent
We consider the classical trajectories $(x^\pm(t), \xi^\pm(t))$ solutions to
\begin{equation}\label{traj}
 \dot x^\pm(t)=\xi^\pm(t),\;\;\dot \xi^\pm(t)=-\nabla
 \lambda_\pm(x^\pm(t)),\;\;x^\pm(0)=x^\pm_0,\;\xi^\pm(0)=\xi_0^\pm. 
 \end{equation}
We have the following result :
\begin{lemma}
Let $\left(x_0^\pm, \xi_0^\pm \right) \in \R^d \times \R^d$. 
\\
\noindent
Under point $(i)$ of Assumption \ref{assumption},
for each $+$ and $-$ trajectory, \eqref{traj} has a unique global, smooth solution $(x^\pm, \xi^\pm) \in C^\infty (\R, \R^d)^2$.
\\
Moreover, the following estimate is satisfied :
\begin{equation}
 \label{trajestimate}
\exists C_0, C_1 >0,  \; |x^\pm(t)| \leq C_0 t, \; \; |\xi^\pm(t)| \leq C_1, \quad \forall t \in \R.
\end{equation}
\end{lemma}
\noindent
The proof of this lemma is based on easy differential inequality arguments and is left to the reader.
\\
\noindent
We denote by~$S^\pm$
the action associated with $(x^\pm(t),\xi^\pm(t))$
\begin{equation}\label{action}
S^\pm(t)=\int_0^t \left( \frac{1}{2} |\xi^\pm(s)|^2-\lambda_\pm(x^\pm(s))\right)\,ds.
\end{equation}
The corresponding energies $E^\pm(t)$ are given by :
$$E^\pm (t)= \dfrac{|\xi^\pm(t)|^2}{2} + \lambda_\pm (x^\pm(t)).$$
These energies are constant along the trajectories :
$$E^\pm (t)=E^\pm (0)= \dfrac{|\xi^\pm_0|^2}{2} + \lambda_\pm (x^\pm_0), \quad \forall t \in \R.$$

\subsection{The ansatz}
\label{ansatz}
We consider the classical trajectories and the action associated with $\lambda_+(x)$ and denote by 
$$Q^+(t) = {\rm Hess} \, \, \lambda_+ (x^+(t)). $$
We consider the function  $u=u(t,y)$  solution to  
 \begin{equation}
\label{profil}
 i\partial_t  u +\frac{1}{2} \Delta u= \frac{1}{2} \left\langle Q^+(t)y;y \right\rangle 
u+\Lambda |u|^{2}u\quad ;\quad 
u(0,y)=a(y),
\end{equation}
and we denote by 
$\vf^\eps$  the function associated with~$u,
x^+,\xi^+, S^+$ by:
\begin{equation}\label{phi}
\varphi^\eps(t,x)=\eps^{-d/4} u
\left(t,\frac{x-x^+(t)}{\sqrt\eps}\right)e^{i\left(S^+(t)+\xi^+(t)
    (x-x^+(t))\right)/\eps}.
\end{equation}
Global existence, conservation of the $L^2-$norm of $u$, and control of its derivatives are 
proved in \cite{Ca-p}. By Corollary 1.11 of \cite{Ca-p}, we actually have
\begin{proposition}
 \label{A}
Let $T>0$, and $a \in \Sch(\R^d)$. Then, for all $k \in \N$, there exists $C=C(T,k)$ such that
$$\forall \alpha, \beta \in \N^d, \; |\alpha|+|\beta|\leq k, \; \|x^\alpha \d_x^\beta u(t) \|_{L^2} \leq C, \quad \forall t \in [0,T].$$
\end{proposition}
\noindent
We will use the following notations:
\begin{notation}
For $p \in \N$, we define the functional spaces $H_\eps^p$ by
\begin{equation*}
 H_\eps^p = \left\lbrace f \in L^2(\R^d), \quad 
 \sum_{|\alpha|\leq p} \| \eps^{|\alpha|} \d_x^\alpha f \|_{L^2}^2< +\infty \right\rbrace
\end{equation*}
For all $f \in H_\eps^p$, we write the associated norm : 
$$\|f\|_{H_\eps^p} = \left(\sum_{|\alpha|\leq p} \| \eps^{|\alpha|} \d_x^\alpha f \|_{L^2}^2\right)^{1/2} $$
\end{notation}
\noindent
We now state the main result of the paper. Of course, if we consider initial data polarized along the other eigenvector,
a similar result is available, with a corresponding ansatz.
\begin{theorem}
\label{main}
Let $T >0$ and $a \in \Sch(\R^d)$. Under assumption \ref{assumption}, consider $\psi^\eps$, 
the exact solution to the Cauchy problem \eqref{NLS0} - \eqref{data}, and $\vf^\eps$, the approximation given by \eqref{phi}.
If we denote by $w^\eps$ the difference $$w^\eps(t,x) = \psi^\eps (t,x) - \vf^\eps(t,x)\chi_+(x),$$
then $w^\eps$ satisfies
\begin{equation*}
\sup_{t \in [0,T]} \| w^\eps (t) \|_{H_\eps^1} \Tend\eps 0 0.
\end{equation*}
\end{theorem}
\begin{remark}
We choose to study a $2 \times 2$ system to simplify notations, but this result can be generalized for a $N \times N$ system, without any 
crossing point. In this case, it is necessary to take time-dependent eigenvectors, to deal with high multiplicities, as it is done in 
\cite{Adiab}, for the case $d=1$ (see \cite{Adiab} and \cite{Hag94} for details about the procedure).
\end{remark}
\begin{remark}
\label{perturbation}
 If we consider initial data which are perturbation of wave packets :
$$\psi^\eps_0(x)= \eps^{-d/4} e^{i\xi_0^+.
    (x-x_0^+)/\eps}a
\left(\frac{x-x_0^+}{\sqrt\eps}\right) \chi_+(x) + \eta^\eps (x), $$
where $\eta^\eps$ satisfies
$$ \|\eta^\eps\|_{L^2(\R^d)} + \|\eps \nabla \eta^\eps\|_{L^2(\R^d)} \leq C \eps^{\gamma_0}, $$
with $\gamma_0 > d/8$, then the approximation of Theorem \ref{main} is still valid (See Remark~\ref{Perturbation2} for details).
\end{remark}
\vspace{3mm}
\noindent
If we assume that for all $k\leq 6 $, we have
\begin{equation}
 \label{B}
\exists C>0, \; \sup_{|\alpha| + |\beta| \leq k} \| x^\alpha \d_x^\beta u(t)\|_{L^2} \leq Ce^{C|t|}, \quad \forall t \in \R,
\end{equation}
it is possible to deal with large time, and to obtain the same result up to a time $T^\eps$ depending on $\eps$:
\begin{theorem}
 \label{loglog}
Let $a \in \Sch(\R^d)$. If for all $k \leq 6$, the estimate \eqref{B} is satisfied, 
then there exists $\eps_0$ such that for all $\eps \in ]0,\eps_0]$,
$$\sup_{t\leq C \log \log \left(\frac{1}{\eps}\right)}\|w^\eps(t)\|_{H_\eps^1} \Tend\eps 0 0.$$
\end{theorem}
\noindent
Besides, for initial data given by Remark \ref{perturbation}, it is possible to prove the same result for large times. 
These points will be discussed after the proof of Theorem \ref{main}.
\smallbreak
\noindent
The behaviour of $u(t)$ for large time is an open question in general. However, there are situation where 
an exponential control of these momenta and derivatives is proved: when $d=1$ or $d\geq 1$ with negative eigenvalues (See Proposition 1.12 of 
\cite{Ca-p}), or if $d\geq 1$, $V(x) = V_\infty$ outside an compact subset $K$ and $x(t) \Tend t \infty \infty$ (see \cite{Ca-p}).
The result must be true in a more general framework. It is possible to prove it under more general conditions on $Q^+$. 
\\
\smallbreak
\noindent
Let us first define Strichartz admissible pairs :
\begin{definition}\label{def:adm}
 A pair $(p,q)$ is {\bf admissible} if $2\leq q
  \leq\frac{2d}{d-2}$ ($2\leq q<
  \infty$ if $d=2$)  
  and 
$$\frac{2}{p}=\delta(q):= d\left( \frac{1}{2}-\frac{1}{q}\right).$$
\end{definition}
The following proposition gives an other situation where the behaviour of the profile is known :
\begin{proposition}
\label{growth}
Let $d=2$ or $3$. Assume $\Lambda \geq 0$ and :
\begin{equation}
\label{assumptionQ}
 \left|\dfrac{d}{dt}Q^+(t)\right| \leq \dfrac{C}{(1+|t|)^{\kappa_0+1}},
\end{equation}
with $\kappa_0 >2$. We consider $u$, the solution to the Cauchy problem \eqref{profil}. 
Then, for all $k \in \N$, the following property is satisfied : 
there exists $C>0$ such that for all admissible pair $(p,q)$, we have
\begin{equation*}
\forall \alpha, \beta \in \N^d, \; |\alpha|+|\beta| \leq k, \quad \|x^\alpha \d_x^\beta u\|_{L^p([0,t],L^q)} \leq C e^{C|t|}, 
\; \forall t \in \R.
\end{equation*}
\end{proposition}
\noindent
Note that with $p = \infty$ and $q=2$, we obtain the property \eqref{B}.
\begin{remark}
\label{rmq:growth}
Let $V$ satisfying Assumption \ref{assumption}. We denote by $E_0$ the energy associated with the trajectories and we introduce 
$\lambda_\infty$ as the following limit (which exists, thanks to Assumption \ref{assumption}) :
$\lim\limits_{|x| \rightarrow \infty} \lambda_+(x) = \lambda_\infty.$ 
\\
\noindent
If $E_0$ is such that $E_0 > \lambda_\infty,$ and
\begin{equation}
 \label{trajinfini}
\lim\limits_{t \rightarrow \infty} |x^+(t)| = + \infty,
\end{equation}
then $Q^+(t)$ satisfies \eqref{assumptionQ}.
The proof of this statement will be sketched in Section~\ref{sec:growth}.
\\
\noindent
Note that \eqref{trajinfini} implies that $E_0\geq \lambda_{\infty}$, so
that the assumption $E_0>\lambda_{\infty}$ is not a very strong one if \eqref{trajinfini} is satisfied.
\end{remark}
\begin{example}
If we consider the potential $V$ introduced in Example \ref{ex} and if we build trajectories 
associated with an eigenvalue of $V$, with $x_0 \neq 0$ and $\xi_0$ such that $\dfrac{|\xi_0|^2}{2} >1$, then, 
it is easy to check that $\lim\limits_{t \rightarrow \infty} |x^+(t)| = + \infty,$ 
that the energy is large enough and so, that $Q^+$ satisfies the property \eqref{assumptionQ}.
\end{example}
\smallbreak
\noindent
We momentarily consider the case $d=1$. This case is considered in \cite{Adiab} with weaker assumptions on the potential and a similar 
approximation in large time is proved (at least for $|t|\leq C\log \log (\eps^{-1})$, with a suitable $C>0$). 
\\
\noindent
However, with a matrix-valued potential under Assumption \ref{assumption}, 
one can obtain it up to a better time $t^\eps=C \log (\eps^{-1})$, at same order as the \textit{Ehrenfest time}. 
Note that in the linear case, this kind of approximation is also valid up to Ehrenfest time (See \cite{BR02} for details).
\begin{theorem}
 \label{1D}
Let $d=1$ and $a \in \Sch(\R)$. Then, there exist $\eps_0>0$ and $C>0$ independent of $\eps$ such that for all $\eps \in ]0,\eps_0]$, 
$$\sup_{t\leq C \log \left(\frac{1}{\eps}\right)}\|w^\eps(t)\|_{H_\eps^1} \Tend\eps 0 0.$$
\end{theorem}

\subsection{Nonlinear superposition}
In this part, we will study the evolution of solutions associated with initial data corresponding to the superposition of two wave packets. 
There are several cases to analyse, depending on whether we choose wave packets polarized along same or different eigenvectors 
(there is actually a technical difference between these cases). 
\\
\noindent
First, we consider two different modes. Let us introduce 
$$ \psi^\eps_0 (x) = \varphi^\eps_+ (0,x) \chi_+(x) + \varphi^\eps_-(0,x) \chi_-(x),$$
where $\varphi^\eps_+$ and $\varphi^\eps_-$ are respectively associated with the modes $+$ and $-$ and 
have the form \eqref{phi}, and $(x^+_0, \xi^+_0)$, $(x^-_0, \xi^-_0)$ 
are phase space points. 
We associate with the phase space points $(x^\pm_0, \xi^\pm_0)$, the classical trajectories $(x^\pm(t), \xi^\pm(t))$, and
the action $S^\pm(t)$ associated with $\lambda_\pm(x)$ such that $$V(x) \chi_\pm (x)= \lambda_\pm (x) \chi_\pm (x). $$
For finite time, we have:
\begin{theorem}
 \label{superposition1}
Let $a_\pm \in \Sch(\R^d)$ and $\varphi^\eps_\pm(0,x)$ as above. We assume 
$$ \Gamma = \inf_{x \in \R} |E_+ - E_- - \left( \lambda_+(x) - \lambda_-(x)\right)| >0.$$
For all $T>0$ (independent of $\eps$), the function 
$$w^\eps (t,x) = \psi^\eps(t,x) - \varphi^\eps_+(t,x) \chi_+(x) - \varphi^\eps_- (t,x) \chi_-(x)$$
satisfies
$$\sup_{t \in [0,T]} \|w^\eps(t)\|_{H_\eps^1} \Tend\eps 0 0. $$
\end{theorem}
We now choose to superpose two wave packets polarized along the same eigenvector, $\chi_+(x)$:
Let $$\psi_0^\eps(x) = \left(\varphi^\eps_1(0,x) + \varphi^\eps_2 (0,x) \right) \chi_+(x), $$
where $\varphi^\eps_1$ and $\varphi^\eps_2$ have the form \eqref{phi}, and $(x_1^+(0),\xi_1^+(0)), (x_2^+(0),\xi_2^+(0))$ 
are phase space points. We assume $$(x_1^+(0), \xi_1^+(0)) \neq (x_2^+(0), \xi_2^+(0)) .$$
Note that without this assumption, the result is obvious with 
$$\varphi^\eps(0,x) = \varphi^\eps_1(0,x) + \varphi^\eps_2 (0,x).$$
\noindent 
We associate with the phase space points, the classical trajectories 
$$(x_1^+(t), \xi_1^+(t)), \; (x_2^+(t), \xi_2^+(t)),$$ 
and the actions $S_1^+(t), S_2^+(t)$, associated with $\lambda_+(x)$.
\\
\noindent
For finite time, we have
\begin{theorem}
 \label{superposition2}
Let $a_j \in \Sch(\R^d)$ and $\varphi^\eps_j(0,x)$ as above, for $j=1,2$, and $T>0$, independent of~$\eps$. Then, the function 
$$ \psi^\eps(t,x) - \left(\varphi^\eps_1(t,x) + \varphi^\eps_2 (t,x)\right) \chi_+(x) $$
satisfies
$$ \sup_{t \in [0,T]} \| w^\eps (t) \|_{H_\eps^1} \Tend\eps 0 0.$$
\end{theorem}
For both cases, infinite time poses a problem that will be discussed in Section \ref{NLsuperposition}. Note that superposition for $d=1$ 
in large time case is proved in \cite{Adiab}, but the arguments are not valid for $d=2$ or $3$ (see Remark \ref{RmqNeps}).

\section{Preliminaries}
\label{preliminaries}
\subsection{About Strichartz estimates}
\label{estimates}
Before beginning the proof, it is crucial to comment the main tool of the proof, the Strichartz estimates.
\\ 
For $d=1$, it is possible to avoid difficulties by using an energy method, 
and the following weighted Gagliardo-Nirenberg inequality to estimate the nonlinearity:
$$\|f\|_{L^\infty} \lesssim \eps^{-1/2} \|f\|_{L^2}^{1/2}\|\eps \d_x f\|_{L^2}^{1/2}, $$
which allows to control the rest (see \cite{Adiab} for the details).
\\
Unfortunatly, this method does not work in our case, with $d=2$ or $3$, since it is $L^2-$supercritical.
In fact, the previous inequality is not valid for $d>1$, there is only the following one:
$$  \|f\|_{L^\infty} \lesssim \eps^{-a} \|f\|_{L^2}^{1-a}\|\eps \nabla f\|_{L^r}^{a},$$
for $r >d$, and $0<a<1$ depending on $r$ and $d$. So it is required to control the $L^r-$norm of $\eps \nabla w^\eps$ for some $r>d$.
An argument using the energy estimate enables us to find a control of the $L^2-$norm of the rest, but this is not sufficient.
Moreover, because of the presence of two modes, it is impossible to choose one, specifically and apply the method of \cite{CF-p}, 
which consists in writing the exact solution as a per\-turbation of the solution of a new equation, involving the Taylor expansion 
of the potential about a point $x^+(t)$ or $x^-(t)$.
\\
For this reason, we need Strichartz estimates. 
In the case of a scalar Schr\"odinger equation, the estimates are available for a scalar external potential, 
with less re\-strictive conditions (the potential can be at most quadratic, see \cite{Fujiwara79} and \cite{Fujiwara} and the discussion 
in \cite{Ca-p}).
\\
In the matrix case, for a potential which is at most quadratic, there is no de\-monstrated Strichartz estimate 
for this kind of matrix-valued Schr\"odinger operator. 
We choose a weaker potential, satisfying point $(i)$ of Assumption \ref{assumption}, and for which Strichartz estimates are available.
\smallbreak
\noindent
We infer the following result from \cite{FerRou}:
\begin{theorem}
 \label{StrichartzEstimates}
Let $(p,q)$, $(p_1, q_1)$, $(p_2, q_2)$ be admissible pairs, such that 
$$q, q_1, q_2 \neq 2d/(d-2), \;\;\; (q, q_1, q_2 < \infty \; \textrm{if} \;d=2).$$ 
Let $I$ be a finite time interval.
Let us introduce $$u^\eps(t)=e^{i\frac{t}{\eps}P(\eps)} u_0 \qquad \textrm{and} \qquad
 v^\eps(t)= \int_{I\cap \left\lbrace s \leq t \right\rbrace}  e^{i \frac{t-\tau}{\eps}P(\eps)}f^\eps(\tau) d\tau.$$
\noindent
$-$ There exists $C=C(q, |I|)$, independent of $\eps$ such that for all $u_0 \in L^2(\R^d)$, we have for all $s \in I$
\begin{equation}
\label{Strichartz1}
\|u^\eps \|_{L^p (I, L^q(\R^d))} \leq C \eps^{-1/p}\|u^\eps(s)\|_{L^2(\R^d)}.
\end{equation}
$-$ There exists $C = C(q_1, q_2, |I|)$, independent of $\eps$ such that for all 
\\
$f^\eps \in L^{p'_2} (I, L^{q'_2}(\R^d))$ we have
\begin{equation}
\label{Strichartz2}
\|v^\eps \|_{L^{p_1} (I, L^{q_1}(\R^d))} \leq C \eps^{-1/p_1 - 1/p_2} \|f^\eps \|_{L^{p'_2} (I, L^{q'_2}(\R^d))}.
\end{equation}
\end{theorem}

\begin{remark}
Let us first remark that the endpoint $(2,2d/(d-2))$ ($(2,\infty)$ if $d=2$) is excluded, as Strichartz estimates in \cite{FerRou} 
are not demonstrated for this pair. 
Besides, in \cite{FerRou}, the authors actually obtain the estimates with a localization. 
In view of point $(iii)$, the eigenprojectors are constant for $x$ large enough, and for this reason, we can drop the localization. 
This point is more explicitly discussed in Remarks 4 and 6 of \cite{FerRou}. 
Finally, let us emphasize that, thanks to the absence of crossing points, we obtain the same Strichartz estimates than in the scalar case, 
without any further loss. The estimates in the general case of \cite{FerRou}, where the eigenvalues might cross, are weaker. 
The procedure to obtain \eqref{Strichartz1} and \eqref{Strichartz2} is sketched in the Appendix~\ref{StrichartzAnnexe}. 
\end{remark}

\subsection{Strategy of the proof}
\label{strategy}
The main difficulty is due to the fact that the projectors do not commute with $P(\eps)$. 
We will adapt ideas of \cite{Adiab} to our situation. 
\\
\noindent
We study the problem for large time, assuming \eqref{B}.
\\
\noindent
We first observe that the function $\varphi^\eps$ satisfies the following equation :
\begin{equation}
 \label{eqphi}
i \eps \d_t \varphi^\eps + \dfrac{\eps^2}{2} \Delta \varphi^\eps - \lambda^+ (x) \varphi^\eps =  \Lambda \eps^{1 + d/2} 
|\varphi^\eps|^2 \varphi^\eps
- \mathcal{R}^\eps (t,x) \varphi^\eps,
\end{equation}
with $\varphi^\eps_0 (x) = \eps^{-d/4} e^{i\xi_0^+.
    (x-x_0^+)/\eps}a
\left(\frac{x-x_0^+}{\sqrt\eps}\right)$, 
for all $x \in \R^d$, and where
\begin{multline*}
 \mathcal{R}^\eps (t,x) = \lambda^+(x)- \lambda^+ (x^+(t)) -  \nabla \lambda^+(x^+(t))(x-x^+(t)) \\ 
-\dfrac{1}{2}\left\langle Q^+(t)(x-x^+(t)); (x-x^+(t)) \right\rangle.
\end{multline*}
We denote by $w^\eps$ the difference between the exact solution and the approximation, 
$$w^\eps (t,x) =  \psi^\eps (t,x) - \varphi^\eps (t,x) \chi^+(x),$$ which satisfies $w^\eps_{|t=0} = 0$ 
(or $w^\eps_{|t=0} = \mathcal{O}(\eps^{\gamma_0})$ in the framework of Remark \ref{perturbation}) and
\begin{equation*}
i \eps \d_t w^\eps (t,x) + \dfrac{\eps^2}{2} \Delta w^\eps (t,x) - V(x) w^\eps (t,x) = 
 \eps \widetilde{NL}^\eps (t,x) + \eps \widetilde{L}^\eps (t,x),
\end{equation*}
where
\begin{align*}
\widetilde{NL}^\eps & = \Lambda \eps^{d/2} \left( |\psi^\eps|^2_{\C^2} \psi^\eps - |\varphi^\eps|^2 \varphi^\eps \chi^+ \right), \\
\widetilde{L}^\eps & =  \eps^{-1} \mathcal{R}^\eps (t,x) \varphi^\eps \chi^+(x) - \dfrac{\eps}{2} \varphi^\eps \Delta \chi^+ 
- \eps \nabla \varphi^\eps d\chi^+\eps \nabla \varphi^\eps.
\end{align*}
Using \eqref{B} and \eqref{eigen}, we can treat the first and the second terms of $\widetilde{L}^\eps$ which gives:
$$ \eps^{-1} \mathcal{R}^\eps \varphi^\eps = \mathcal{O}\left(\sqrt{\eps} e^{Ct}\right),$$
and
$$\eps \varphi^\eps \Delta \chi^+ = \mathcal{O} (\eps e^{Ct}).$$
Observing that the last term satisfies the following equality:
\begin{equation*}
\eps \nabla \varphi^\eps = i \xi^+(t) \varphi^\eps + \mathcal{O} \left( \sqrt{\eps} e^{Ct} \right), 
\qquad \textrm{in} \; L^2(\R^d). 
\end{equation*}
and using \eqref{trajestimate}, we infer that the last term of $\widetilde{L}^\eps$ brings a difficulty, 
as it a priori presents an $\mathcal{O}(1)$ contribution. 
This is an obstruction to prove that $w^\eps$ is small when $\eps$ tends to zero.
Therefore, we have to introduce a correction term to $w^\eps$, to get rid of this difficulty.
\\
We denote by $g^\eps$, the function solving the Schr\"odinger equation
\begin{equation*}
i \eps \d_t g^\eps(t,x) + \dfrac{\eps^2}{2} \Delta g^\eps(t,x) - \lambda_-(x) g^\eps(t,x) = r(t,x) \varphi^\eps(t,x) 
\quad ; \quad g^\eps(0,x)=0,
\end{equation*}
where
\begin{equation}
 \label{r}
r(t,x) = - i \left\langle d\chi^+(x)\xi^+(t) , \chi^-(x) \right\rangle_{\C^2}. 
\end{equation}
Let us remark that the above-mentionned quantity is bounded with bounded derivatives, thanks to \eqref{eigen} and \eqref{trajestimate}:
\begin{equation}
 \label{r'}
\forall p \in \N, \forall \alpha \in \N^d, \exists C >0, \forall t \in \R, \forall x \in \R^d, \; | \d_t^p \d_x^\alpha r(t,x)| \leq C.
\end{equation}
To deal with the nonlinearity, we need to control the $L^2-$norm and $L^4-$norm of the correction term $g^\eps(t)$ 
and its derivatives; we have the following proposition, which holds for large time; it will be proved in Section \ref{correc}.
\begin{proposition}
 \label{L2}
Assuming \eqref{B}, for $p \in \N$, there exists $C=C(p)$ such that
$$ \|g^\eps(t)\|_{H_\eps^{p}} \lesssim e^{Ct}, \quad \forall t\geq 0.$$
Moreover, for all $\alpha \in \N^d$, there exists $C=C(\alpha)$ such that
$$\|\eps^{|\alpha|} \d_x^\alpha g^\eps(t) \|_{L^4} \lesssim \eps^{-d/4}e^{Ct}, \quad \forall t \geq 0.$$
\end{proposition}
\smallbreak
We now set 
 $$ \theta^\eps(t,x) = w^\eps (t,x) + \eps g^\eps (t,x) \chi_-(x).$$
This function then solves
\begin{equation}
 \label{theta}
i \eps \d_t \theta^\eps(t,x) + \dfrac{\eps^2}{2} \Delta \theta^\eps(t,x) - V(x) \theta^\eps(t,x) = 
\eps NL^\eps(t,x) + \eps L^\eps(t,x), \qquad \theta^\eps(0,x) = 0,
\end{equation}
with
\begin{equation}
 \label{NL}
NL^\eps = \Lambda \eps^{d/2} \left(|\varphi^\eps \chi_+ + \theta^\eps - \eps g^\eps \chi_-|^2_{\C^2}
(\varphi^\eps \chi_+ + \theta^\eps - \eps g^\eps \chi_-) - |\varphi^\eps|^2\varphi^\eps\chi_+ \right) 
\end{equation}
\begin{align}
\notag
L^\eps & = \widetilde{L}^\eps + \left(i \eps \d_t g^\eps + \dfrac{\eps^2}{2} \Delta g^\eps - \lambda_-(x) g^\eps \right)\chi_-
 +\eps^2  d\chi_-\nabla g^\eps + \dfrac{\eps^2}{2} g^\eps \Delta \chi_-, \\
\label{L}
 & = \mathcal{O}(\sqrt{\eps}e^{Ct}) +\eps^2  d\chi_-\nabla g^\eps + \dfrac{\eps^2}{2} g^\eps \Delta \chi_-,
\end{align}
where the $\mathcal{O}(\sqrt{\eps}e^{Ct})$ holds in $L^2$. 
Then, using the control of the eigenvectors \eqref{eigen} and the control of $g^\eps$, given by Proposition \ref{L2}, we infer
$$\| L^\eps(t)\|_{L^2} = \mathcal{O}(\sqrt{\eps}e^{Ct}),$$
where $C$ is independent of $\eps$.
\begin{remark}
\label{finite}
In view of Proposition \ref{A}, for all $T>0$, there exists a constant $C>0$ independent of $\eps$, such that
$$\|L^\eps(t)\|_{L^2}\leq C \eps^{1/2}, \quad \forall t \in [0,T].$$
Besides, we can write Proposition \ref{L2} for finite time intervals, which gives : 
\\
\noindent
Set $T>0$, then for $p \in \N$, there exists $C = C(T,p)$ such that
$$ \| g^\eps(t)\|_{H^p_\eps} \leq C, \quad \forall t \in [0,T];$$
and for $\alpha \in \N^d$, there exists $C=C(\alpha, T)$ such that  
$$ \|\eps^{|\alpha|}\d_x^\alpha g^\eps(t) \|_{L^4} \leq C \eps^{-d/4}, \forall t \in [0,T].$$
These estimates will be useful to deal with finite time intervals.
\end{remark}
\noindent
The proof of Proposition \ref{L2}, about $g^\eps$, is presented in the following section.
Secondly, the final step of the main proof, analysing the behaviour of $\theta^\eps$ as $\eps$ goes to zero is studied in Section 
\ref{last} for finite time case, and in Section \ref{largetimecase} for infinite times. 
Then, the behaviour of the profile $u$, important for large time case, is discussed in Section \ref{sec:growth} 
and the analysis of the one dimension case is done in Section~\ref{d1}.
Finally, Section \ref{NLsuperposition} is devoted to the proof 
of superposition results (Theorems \ref{superposition1} and \ref{superposition2}).

\section{Proof of the main results}
\subsection{Estimate of the correction term}
\label{correc}
In this section, we prove Proposition~\ref{L2}, assuming that we have the exponential control \eqref{B}. The proof of Remark \ref{finite} 
follows the same lines, in view of Proposition \ref{A}.
\\
\noindent 
In view of the control of the classical trajectories and of the profile $u$, for all $p \in \N$, there exists $C=C(p)$, such that
\begin{equation}
\label{ExpPhi}
 \| \varphi^\eps (t) \|_{H_\eps^p} \lesssim e^{Ct}, \quad \forall t \geq 0.
\end{equation}
Besides, if we have the exponential control of $u$ and of its derivatives, stated in \eqref{B}, we note that 
$\d_y^\alpha u(t,.)$ is in $L^\infty$ for all $\alpha \in \N^d$. From this estimate, we infer
\begin{equation}
 \label{LInfPhi}
\forall \alpha \in \N^d, \; \exists C = C(\alpha), \quad 
\| \eps^{|\alpha|} \d_x^\alpha \varphi^\eps(t)\|_{L^\infty} \lesssim \eps^{-d/4} e^{Ct}.
\end{equation}
Write $U^\eps_\pm (t)=e^{i\frac{t}{\eps}p_\pm(\eps)}$, 
the semi-group associated with the operator $$p_\pm(\eps):=-\dfrac{\eps^2}{2} \Delta + \lambda_\pm(x).$$
We observe that for $p \in \N$, there exists a constant $K=K(p)$ such that
\begin{equation}
\label{ExpPropagateurs}
 \| U^\eps_\pm (t) \|_{\mathcal{L}(H_\eps^p)} \lesssim e^{K|t|}.
\end{equation}
For $\lambda_+$ and $\lambda_-$ as in Assumption \ref{assumption}, 
we have the following lemma, which will be needed to estimate the correction term $g^\eps$. 
Note that the crucial point is that the eigenvalues satisfy Point $(ii)$ of Assumption \ref{assumption}.
\begin{lemma}
 \label{propagateurs}
For $T>0$, there exists a constant $C$ such that
$$\forall t \in [0,T], \quad \forall p \in \N, \quad 
\left\|\dfrac{1}{i\eps} \int_0^t U_+^\eps (-s) U_-^\eps (s) ds \right\|_{\mathcal{L}(H_\eps^{p+1},H_\eps^{p})}\leq C e^{Ct}. $$
The same estimate remains valid if we permute $U_+^\eps$ and $U_-^\eps$.
\end{lemma}
\begin{proof}
This proof follows \cite{Adiab}, Lemma 3.1.
We first notice that
$$U_+^\eps(-t) U_-^\eps(t) = i \eps U_+^\eps(-t)\left(\lambda_+ 
- \lambda_- \right)^{-1} U_+^\eps(t) \d_t \left(U_+^\eps(-t)U_-^\eps(t)\right).$$
Then, by integration by parts
\begin{align*}
 \dfrac{1}{i\eps} \int_0^t U_+^\eps (-s) U_-^\eps (s) ds = 
& \left[U_+^\eps(-s)\left(\lambda_+ - \lambda_- \right)^{-1} U_-^\eps (s)\right]_0^t \\
& - \int_0^t \d_s \left(U_+^\eps(-s)\left(\lambda_+ - \lambda_- \right)^{-1}U_+^\eps(s) \right) U_+^\eps(-s)U_-^\eps(s) ds. 
\end{align*}
Write $\gamma = \left(\lambda_+ - \lambda_- \right)^{-1}$. 
Using \eqref{eigen}, and point $(ii)$ of Assumption \ref{assumption}, we infer that $\gamma$ is bounded with bounded derivatives:
\begin{equation*}
\forall \alpha \in \N^d, \; \exists C>0, \quad | \d_x^\alpha \gamma(x)| \leq C.
\end{equation*}
Since the propagators map continuously $H_\eps^{p}$ into itself, uniformly with respect to ~$\eps$, we infer
$$\forall p \in \N,\; \exists C=C(p), 
  \quad \| \left[U_+^\eps(-s) \gamma U_-^\eps(s) \right]_0^t\|_{\mathcal{L}(H_\eps^{p},H_\eps^{p})}\leq C(p)e^{C(p)t}.$$
Besides, we have
\begin{align*}
 \d_s \left(U_+^\eps(-s)\left(\lambda_+ - \lambda_- \right)^{-1}U_+^\eps(s) \right) & = 
 \dfrac{1}{i\eps} U_+^\eps(-s) \left[ -\dfrac{\eps^2}{2} \Delta, \gamma\right] U_+^\eps(s), \\
& = U_+^\eps(-s) \left( i \d_x \gamma(x) \eps \d_x + i \eps \d_x^2 \gamma(x)  \right) U_+^\eps(s).
\end{align*}
Combining :
\begin{align*}
 & \| U_+^\eps(-s) \d_x \gamma(x) \eps \d_x U_+^\eps(s) \|_{\mathcal{L}(H_\eps^{p+1}, H_\eps^{p})} \lesssim e^{Cs},  \\
 \textrm{and} \quad & \| U_+^\eps(-s) \d_x^2 \gamma(x)U_+^\eps(s) \|_{\mathcal{L}(H_\eps^{p}, H_\eps^{p})} \lesssim e^{Cs},
\end{align*}
we complete the proof.
\end{proof}
\noindent
We now prove Proposition \ref{L2}:
\begin{proof}[Proof of Proposition \ref{L2}]
We follow the steps of \cite{Adiab}, Proposition 3.2. 
\\
\noindent
Let us write $\widetilde{\varphi^\eps} (t,x) = r(t,x) \varphi^\eps(t,x)$, then we have
$$\left(i \eps \d_t + \dfrac{\eps^2}{2} \Delta - \lambda_+(x)\right)\widetilde{\varphi^\eps} =
\underbrace{\Lambda \eps^{1+d/2} |\varphi^\eps|^2 r \varphi^\eps - \mathcal{R^\eps} r \varphi^\eps + i\eps \d_t r \varphi^\eps}_{\eps f^\eps} 
+ \dfrac{\eps^2}{2}\left[\Delta, r \right]\varphi^\eps; $$
with $f^\eps = \Lambda \eps^{d/2} |\varphi^\eps|^2 r \varphi^\eps +  i\d_t r \varphi^\eps 
+ \dfrac{\eps}{2}\left[\Delta, r \right]\varphi^\eps - \eps^{-1} \mathcal{R^\eps} r \varphi^\eps$.
By Duhamel's formula, we obtain
$$\widetilde{\varphi^\eps}(t) = U_+^\eps(t)\widetilde{\varphi^\eps}(0) - i \int_0^t U_+^\eps(t-s) f^\eps(s) ds.$$
We deduce
\begin{align*}
g^\eps(t) & = \dfrac{1}{i\eps} \int_0^t U_-^\eps(t-s) \widetilde{\varphi^\eps}(s)ds,\\
          & = \dfrac{1}{i\eps} \int_0^t U_-^\eps(t-s)U_+^\eps (s)\widetilde{\varphi^\eps}(0) ds - 
 \int_0^t \dfrac{1}{\eps} \int_0^s U_-^\eps(t-s) U_+^\eps(s-\tau) f^\eps(\tau) d\tau ds.
\end{align*}
We write $U_-^\eps(t-s) = U_-^\eps(t-\tau)U_-^\eps(\tau-s)$ and applying Fubini's theorem, we obtain
\begin{multline*} 
 g^\eps(t) = \dfrac{1}{i\eps} \int_0^t U_-^\eps(t-s)U_+^\eps (s)\widetilde{\varphi^\eps}(0) ds \\
-\int_0^t \dfrac{1}{\eps}U_-^\eps(t-\tau) \int_\tau^t U_-^\eps(\tau-s) U_+^\eps(s-\tau) f^\eps(\tau) ds d\tau.
\end{multline*}
Using Lemma \ref{propagateurs}, we have
$$\|g^\eps\|_{H_\eps^{p}} \lesssim e^{Ct} + \int_0^t e^{C(t-\tau)} \|f^\eps(\tau)\|_{H_\eps^{p+1}} d\tau. $$
It remains to study $f^\eps$. We write $f^\eps = f_1^\eps + f_2^\eps$, with
$$f_1^\eps = i \d_t r \varphi^\eps + \dfrac{\eps}{2}\left[\Delta, r \right]\varphi^\eps - \eps^{-1}\mathcal{R^\eps} r \varphi^\eps, 
\quad \textrm{and} \quad 
f_2^\eps = \Lambda \eps^{d/2} |\varphi^\eps|^2 r \varphi^\eps.$$
By \eqref{r'} and \eqref{ExpPhi}, it is straightforward that
$$\|f_1^\eps(t)\|_{H_\eps^{p+1}} \lesssim e^{Ct},$$
provided that $(Exp)_{p+1+3}$ is satisfied, to deal with the term $\eps^{-1}\mathcal{R^\eps} r \varphi^\eps$. 
\\
\noindent
Besides
\begin{align*}
 \|f_2^\eps(t)\|_{H_\eps^{p+1}} & = \eps^{d/2} \|r |\varphi^\eps|^2 \varphi^\eps \|_{H_\eps^{p+1}}, \\
& \lesssim \eps^{d/2} \|\varphi^\eps \|_{H_\eps^{p+1}} 
\sup_{0\leq |\alpha| \leq p+1}\|\eps^{|\alpha|} \d_x^\alpha \varphi^\eps\|_{L^\infty}^2 \\
& \lesssim \eps^{d/2}(\eps^{-d/4})^2 e^{Ct} \lesssim e^{Ct},
\end{align*}
where we have used the control of $\varphi^\eps$, \eqref{ExpPhi} and \eqref{LInfPhi}, and the proof is complete.
\\
The proof of the other estimate is based on a Sobolev embedding and on H\"older inequality.
\\
\noindent
Let $p\in \N$ and $\alpha \in \N^d$, such that $|\alpha|\leq p$. 
We first notice that $H^{d/4}(\R^d) \hookrightarrow L^4(\R^d)$, and infer
$$\| \eps^{|\alpha|} \d_x^\alpha g^\eps(t) \|_{L^4} \leq c \|\eps^{|\alpha|} \d_x^\alpha g^\eps(t) \|_{H^{d/4}}.$$
Besides, we introduce the following Lebesgue exponents
$$q=\dfrac{4}{d}, \quad r=\dfrac{4}{4-d}.$$
Using the interpolation inequality,
$$\|\eps^{|\alpha|} \d_x^\alpha g^\eps(t) \|_{H^{d/4}}\leq 
c \|\eps^{|\alpha|} \d_x^\alpha g^\eps(t)\|_{L^2}^{1-d/4} \|\eps^{|\alpha|} \d_x^\alpha g^\eps(t) \|_{H^1}^{d/4},$$
we write
\begin{align*}
\|\eps^{|\alpha|} \d_x^\alpha g^\eps(t) \|_{H^1} & = \|\eps^{|\alpha|} \d_x^\alpha g^\eps(t)\|_{L^2} 
+ \|\d_x \left( \eps^{|\alpha|} \d_x^\alpha g^\eps(t)\right)\|_{L^2}, \\
& \lesssim \| g^\eps(t) \|_{H_\eps^p} + \eps^{-1} \|g^\eps(t) \|_{H_\eps^{p+1}},
\end{align*}
and, using the first estimate of Proposition \ref{L2}, the proof is complete.
\end{proof}

\subsection{End of the proof of Theorem \ref{main}.}
\label{last}
We now prove Theorem \ref{main}. In this section, we consider finite time intervals, we will use the estimates of Proposition~\ref{A}, 
which imply Remark \ref{finite}. 
We divide the proof into three steps : first, we will analyse 
a Strichartz norm of $\varphi^\eps$, which will lead to introduce a bootstrap argument. Then, using the bootstrap assumption, we will prove 
the theorem, before checking the validity of the bootstrap in the final step.
\smallbreak
\noindent
\textbf{Step one :} 
\\
\noindent 
We recall the equation satisfied by the rest $\theta^\eps$, \eqref{theta} :
\begin{equation*}
 i \eps \d_t \theta^\eps (t,x) +  \dfrac{\eps^2}{2} \Delta \theta^\eps (t,x) - V(x) \theta^\eps (t,x) =
\eps NL^\eps + \eps L^\eps \quad ; \quad \theta^\eps (0,x) = 0,
\end{equation*}
where $ NL^\eps$ and $ L^\eps$ are defined in \eqref{NL} and \eqref{L}, respectively.
The Duhamel formula gives
\begin{align*}
 \theta^\eps(t+ \tau) & = e^{i \frac{\tau}{\eps}P(\eps)} \theta^\eps (t) 
-i \int_{t}^{t+\tau} e^{\frac{i}{\eps}(t+\tau-s)P(\eps)} NL^\eps (s) ds \\
& -i \int_{t}^{t+\tau} e^{\frac{i}{\eps}(t+\tau-s)P(\eps)} L^\eps (s) ds.
\end{align*}
We introduce the following Lebesgue exponents:
$$ p = \dfrac{8}{d} \quad ; \quad q = 4\quad ; \quad \sigma = \dfrac{8}{4-d} . $$
Then, $(p,q)$ is admissible, and
$$\dfrac{1}{p'} = \dfrac{2}{\sigma} + \dfrac{1}{p} \quad ; \quad \dfrac{1}{q'} = \dfrac{3}{4} = \dfrac{2}{q} + \dfrac{1}{q}. $$
Let $t \geq 0, \tau >0$ and $I = [t, t+\tau]$. Strichartz estimates of Theorem \ref{StrichartzEstimates} yield
\begin{align*}
 \| \theta^\eps \|_{L^p(I, L^q)}  \lesssim & \; \eps^{-1/p} \| \theta^\eps (t)\|_{L^2} + \eps^{- 1/p} \| L^\eps \|_{L^1(I, L^2)} \\
& + \eps^{- 2/p} \| NL^\eps  \|_{L^{p'}(I, L^{q'})}.
\end{align*}
In view of the pointwise estimate
\begin{multline}
 \Bigl| |\varphi^\eps \chi_+ + \theta^\eps - \eps g^\eps \chi_-|^2_{C^2}
(\varphi^\eps \chi_+ + \theta^\eps - \eps g^\eps \chi_-) - |\varphi^\eps|^2\varphi^\eps\chi_+ \Bigr| 
  \\
 \lesssim \left( |\varphi^\eps|^2 + |\theta^\eps|^2 + \eps^2 |g^\eps |^2 \right) \left( |\theta^\eps| + |\eps g^\eps| \right),
\end{multline}
and using H\"older inequality, we infer
\begin{align*}
\| \theta^\eps \|_{L^p(I, L^q)}  \lesssim & \; \eps^{-1/p} \| \theta^\eps (t)\|_{L^2} + \eps^{- 1/p} \| L^\eps \|_{L^1(I, L^2)} 
+ \eps^{d/2 - 2/p} \left(\|\varphi^\eps \|_{L^\sigma(I, L^q)}^2 \right. \\
& \left. + \|\theta^\eps \|_{L^\sigma(I, L^q)}^2  + \eps^2 \|g^\eps \|_{L^\sigma(I, L^q)}^2 \right)
\left( \|\theta^\eps \|_{L^p(I, L^q)} + \eps \|g^\eps \|_{L^p(I, L^q)} \right). 
\end{align*}
We have
\begin{equation}
\label{phiL4}
 \| \varphi^\eps (t)\|_{L^4_x} = \eps^{-d/8} \|u(t)\|_{L^4_y} \leq C(T) \eps^{-d/8},
\end{equation}
with $y= \dfrac{x-x^+(t)}{\sqrt{\eps}}$. Besides, using Proposition \ref{L2} again, we obtain the estimate
\begin{equation}
 \label{g}
\eps^2 \|g^\eps(t)\|_{L^4}^2 \leq C(T) \eps^{2-d/2},
\end{equation}
with $2- d/2 \geq 1/2$.
\\
\noindent
Therefore, it is natural to perform a bootstrap argument assuming, say
\begin{equation}
 \label{bootstrap}
\|\theta^\eps(t)\|_{L^4} \lesssim \;\eps^{-d/8}.
\end{equation}
In the rest of the proof, we will not mention dependance in $T$ of the terms.
\smallbreak
\noindent
\textbf{Step two :} 
\\
\noindent
In this step, we assume that \eqref{bootstrap} holds on $[0,T]$ and show :
\begin{equation}
\label{3.10'}
 \| \theta^\eps (t) \|_{L^2} \leq C(T) \eps^{1/2},
\end{equation}
and
\begin{equation}
 \label{3}
\| \theta^\eps \|_{L^p([0,t],L^q)} \lesssim \eps^{1/2 - d/8}.
\end{equation}
\noindent
As long as \eqref{bootstrap} holds, we have for all $s \in I$ :
\begin{multline*}
 \| \theta^\eps \|_{L^p(I, L^q)} \lesssim \; \eps^{-1/p} \| \theta^\eps (s)\|_{L^2}\\ 
 + \eps^{- 1/p} \| L^\eps \|_{L^1(I, L^2)} + \tau^{2/\sigma} \left( \| \theta^\eps \|_{L^p(I, L^q)} 
+ \eps \| g^\eps \|_{L^p(I, L^q)} \right). 
\end{multline*}
where we have used \eqref{phiL4}, \eqref{g} and \eqref{bootstrap}, with
$$\dfrac{d}{2} - \dfrac{2}{p} - \dfrac{2d}{8} = 0.$$
Integrating in $s$, between $t$ and $t+\tau$, we get
\begin{multline*}
 \| \theta^\eps \|_{L^p(I, L^q)}   \lesssim \; \eps^{-1/p} \tau^{-1} \| \theta^\eps\|_{L^1(I,L^2)}\\ 
 + \eps^{- 1/p} \| L^\eps \|_{L^1(I, L^2)} + \tau^{2/\sigma} \left( \| \theta^\eps \|_{L^p(I, L^q)} 
+ \eps \| g^\eps \|_{L^p(I, L^q)} \right). 
\end{multline*}
We choose $\tau \ll 1$ and from now, $\tau >0$ is fixed, so that $\tau^{-1}$ is a constant.
We recover the interval $[0,T]$ with a finite number of intervals of the form $[j\tau, (j+1)\tau]$ and obtain
\begin{equation}
 \label{1}
\| \theta^\eps \|_{L^p([0,T], L^q)}  \lesssim \eps^{-1/p} \| \theta^\eps\|_{L^1([0,T],L^2)} + \eps^{- 1/p} \| L^\eps \|_{L^1([0,T], L^2)} 
+ \eps \| g^\eps \|_{L^p([0,T], L^q)}.
\end{equation}
Using Strichartz estimates again, and previous estimates, we have, for $0 \leq t \leq T$,
\begin{align}
\notag
 \| \theta^\eps \|_{L^\infty([0,t],L^2)}& \lesssim \| L^\eps \|_{L^1([0,t],L^2)} + \eps^{-1/p} \| NL^\eps \|_{L^{p'}([0,t],L^{q'})}, \\
\label{***}
& \lesssim \| L^\eps \|_{L^1([0,t], L^2)} 
+ \eps^{d/2 - 1/p} \left(\|\varphi^\eps \|_{L^\sigma([0,t], L^q)}^2 + \|\theta^\eps \|_{L^\sigma([0,t], L^q)}^2\right. \\
\notag
& \quad \left.  + \eps^2 \|g^\eps \|_{L^\sigma([0,t], L^q)}^2 \right) 
\left( \|\theta^\eps \|_{L^p([0,t], L^q)} + \eps \|g^\eps \|_{L^p([0,t], L^q)} \right).
\end{align}
Thanks to \eqref{phiL4}, \eqref{g}, \eqref{1}, and under \eqref{bootstrap}, we obtain
\begin{align*}
\| \theta^\eps \|_{L^\infty([0,t],L^2)}& \lesssim \| L^\eps \|_{L^1([0,t],L^2)} + \eps^{1+1/p} \|g^\eps \|_{L^p([0,t], L^q)}
 + t^{2/\sigma}\| \theta^\eps \|_{L^1([0,t],L^2)},\\
& \lesssim \| L^\eps \|_{L^1([0,t],L^2)}+ \eps^{1-d/8} + \| \theta^\eps \|_{L^1([0,t],L^2)}.
\end{align*}
We use the following estimate, given in Remark \ref{finite}, for $t \in [0,T]$~: 
$$\|L^\eps\|_{L^1([0,t],L^2)}\lesssim \eps^{1/2}.$$
We notice that $1/2 \leq 1-d/8$, for $d=2,3$, and we infer
\begin{equation}
 \label{2}
\| \theta^\eps \|_{L^\infty([0,t],L^2)} \lesssim \eps^{1/2}
 + \| \theta^\eps \|_{L^1([0,t],L^2)}.
\end{equation}
By Gronwall Lemma, we obtain the estimate \eqref{3.10'}.
Combining \eqref{1} and \eqref{2}, we obtain the announced estimate \eqref{3}, under \eqref{bootstrap}, which concludes this step.
\smallbreak
\noindent
\textbf{Step three:}  
\\
\noindent
It remains to check how long the bootstrap assumption \eqref{bootstrap} holds. 
For this, we look for a control of $\theta^\eps(t)$ in $H_\eps^1$.
We differentiate the system \eqref{theta} with respect to $x$, and we find
$$\left\lbrace
\begin{array}{rl}
 i \eps \d_t (\eps \nabla \theta^\eps)+ \dfrac{\eps^2}{2} \Delta (\eps \nabla \theta^\eps) - V(x) \eps \nabla \theta^\eps & = 
\eps \nabla V(x) \theta^\eps + \eps^2 \nabla NL^\eps + \eps^2 \nabla L^\eps, \\
\eps \nabla \theta^\eps (0,x)& =0.
\end{array}\right.
$$
Using Strichartz estimates again, we find
\begin{align*}
\| \eps \nabla \theta^\eps \|_{L^p(I, L^q)}  \lesssim & \; \eps^{-1/p} \| \eps \nabla \theta^\eps (s)\|_{L^2} 
+ \eps^{- 1/p} \| \nabla V \theta^\eps \|_{L^1(I, L^2)} \\
&\quad  + \eps^{- 1/p} \| \eps \nabla L^\eps \|_{L^1(I, L^2)} 
+ \eps^{- 2/p}\| \eps \nabla NL^\eps \|_{L^{p'}(I, L^{q'})}.
\end{align*}
We observe that, thanks to Assumption \ref{assumption}, $|\nabla V(x)| \leq C$. Besides, by Remark~\ref{finite}, we have
\begin{equation}
\label{3.11'}
\|\eps \nabla L^\eps(t)\|_{L^2} \lesssim \sqrt{\eps}, \quad \textrm{for} \; 0 \leq t \leq T.
\end{equation}
The only point remaining concerns the term $\eps \nabla NL^\eps$. We have
\begin{multline*}
| \eps \nabla NL^\eps | \lesssim  \eps^{d/2} \left(|\varphi^\eps|^2 + |\theta^\eps|^2 + \eps^2 |g^\eps|^2\right)
 \left(|\eps \nabla \theta^\eps|+ \eps^2|\nabla (g^\eps \chi_-)|\right) \\
+ \eps^{d/2} |\eps \nabla \varphi^\eps| |\varphi^\eps| \left(|\theta^\eps|+\eps|g^\eps|\right).
\end{multline*}
We notice that 
$$ \| \eps \nabla \varphi^\eps (t) \|_{L^4}\lesssim \eps^{1/2-d/8} \|\nabla u(t)\|_{L^4}+ \eps^{-d/8}\|u(t)\|_{L^4} \lesssim 
C(T)\eps^{-d/8},$$
using Proposition \ref{A}, with $1/2-d/8>0$, we infer by \eqref{phiL4}
\begin{equation}
 \label{L2phi}
\| \varphi^\eps \times \eps\nabla \varphi^\eps (t)\|_{L^2} \lesssim \eps^{-d/4}.
\end{equation}
Then, we can write, thanks to H\"older inequality :
\begin{align*}
 \eps^{-2/p}\| \eps \nabla NL^\eps \|_{L^{p'}(I, L^{q'})} & \lesssim \; \eps^{d/2 - 2/p} 
\left( \|\varphi^\eps \|_{L^\sigma(I,L^q)}^2 + \| \theta^\eps \|_{L^\sigma(I,L^q)}^2 + \eps^2 \|g^\eps \|_{L^\sigma(I,L^q)}^2 \right)\\
&  \left( \| \eps \nabla \theta^\eps\|_{L^p(I, L^q)} + \eps^2 \| \nabla (g^\eps \chi_-)\|_{L^p(I, L^q)} \right) \\
  + & \eps^{d/2 - 2/p}\| \varphi^\eps \; \eps\nabla \varphi^\eps\|_{L^{\sigma/2}(I,L^2)}
\left(\| \theta^\eps\|_{L^p(I, L^q)} + \eps \|g^\eps\|_{L^p(I, L^q)} \right)
\end{align*}
The first part is handled as before, using estimates \eqref{phiL4}, \eqref{g} and \eqref{bootstrap}. 
For the second part, using \eqref{3} instead of \eqref{bootstrap}, we find
\begin{align*}
\eps^{-2/p} \| \eps \nabla NL^\eps \|_{L^{p'}(I, L^{q'})} & \lesssim \tau^{2/\sigma} \| \eps \nabla \theta^\eps\|_{L^p(I, L^q)} + 
\tau^{2/\sigma}\eps^{1/2-d/8},
\end{align*}
where we have used that $d/2 - 2/p - d/4 = 0$. We infer for $s \in I$ :
\begin{align*}
 \| \eps \nabla \theta^\eps \|_{L^p(I, L^q)}  \lesssim & \; \eps^{-1/p} \| \eps \nabla \theta^\eps (s)\|_{L^2} 
+ \eps^{- 1/p} \| \theta^\eps \|_{L^1(I, L^2)}+ \eps^{- 1/p} \| \eps \nabla L^\eps \|_{L^1(I, L^2)} \\
& \quad + \tau^{2/\sigma}\eps^{1/2-d/8} +\tau^{2/\sigma} \| \eps \nabla \theta^\eps\|_{L^p(I, L^q)}.
\end{align*}
By integration on $I$, we obtain :
\begin{align*}
 \| \eps \nabla \theta^\eps \|_{L^p(I, L^q)}  \lesssim & \; \eps^{-1/p} \tau^{-1}\| \eps \nabla \theta^\eps\|_{L^1(I,L^2)} 
+ \eps^{- 1/p} \| \theta^\eps \|_{L^1(I, L^2)} \\
& \quad + \eps^{- 1/p} \| \eps \nabla L^\eps \|_{L^1(I, L^2)}+ \tau^{2/\sigma}\eps^{1/2-d/8} 
+\tau^{2/\sigma} \| \eps \nabla \theta^\eps\|_{L^p(I, L^q)}.
\end{align*}
We choose $\tau$ sufficiently small to absorb the last term, and repeating this procedure a finite number of times, to recover $[0,T]$, 
we obtain
\begin{equation}
\label{4}
 \| \eps \nabla \theta^\eps \|_{L^p([0,T], L^q)} \lesssim \eps^{-1/p}\| \eps \nabla \theta^\eps\|_{L^1([0,T]L^2)} + \eps^{1/2-d/8},
\end{equation}
where we have used \eqref{3.10'}, \eqref{3.11'} and that 
$$ \| \theta^\eps \|_{L^1([0,T], L^2)} + \| \eps \nabla L^\eps \|_{L^1([0,T], L^2)}  \lesssim \eps^{1/2}.$$
Now, for $t \in [0,T]$, Strichartz estimates yield :
\\[2mm]
$\| \eps \nabla \theta^\eps \|_{L^\infty([0,t], L^2)}$
\begin{align*}
 \lesssim & \; \| \nabla V \theta^\eps \|_{L^1([0,t], L^2)} 
+  \| \eps \nabla L^\eps \|_{L^1([0,t], L^2)}+ \eps^{- 1/p}\| \eps \nabla NL^\eps \|_{L^{p'}([0,t], L^{q'})} \\
 \lesssim & \; \| \theta^\eps \|_{L^1([0,t], L^2)} + \| \eps \nabla L^\eps \|_{L^1([0,t], L^2)} \\
& \quad + \eps^{d/2-1/p} \left( \|\varphi^\eps \|_{L^\sigma([0,t],L^q)}^2 + \| \theta^\eps \|_{L^\sigma([0,t],L^q)}^2 
+ \eps^2 \|g^\eps \|_{L^\sigma([0,t],L^q)}^2 \right)\\
& \quad  \left( \| \eps \nabla \theta^\eps\|_{L^p([0,t], L^q)} + \eps^2 \| \nabla (g^\eps \chi_-)\|_{L^p([0,t], L^q)} \right) \\
& \quad + \eps^{d/2 - 1/p}\| \varphi^\eps \eps\nabla \varphi^\eps\|_{L^{2/\sigma}([0,t],L^2)} 
  \left(\| \theta^\eps\|_{L^p([0,t], L^q)} + \eps \|g^\eps\|_{L^p([0,t], L^q)} \right), \\
\lesssim & \; t^{2/\sigma} \| \theta^\eps \|_{L^1([0,t], L^2)} + \| \eps \nabla L^\eps \|_{L^1([0,t], L^2)} \\
& \quad + \eps^{d/2 - 1/p - d/4} \| \eps \nabla \theta^\eps\|_{L^p([0,t], L^q)} + \eps^{d/2 - 1/p - d/4} \times \eps^{1/2-d/8},
\end{align*}
where we have used \eqref{phiL4}, \eqref{g}, \eqref{3}, \eqref{L2phi}; and where the powers of $\eps$ given 
by the correction term $g^\eps$ are not written as they are better than the powers above. 
Using \eqref{3.10'}, \eqref{3.11'} and \eqref{4}, we now have
\begin{align}
\notag
 \| \eps \nabla \theta^\eps \|_{L^\infty([0,t], L^2)} & \lesssim  \eps^{d/8} 
\times \left(\eps^{-1/p} \| \eps \nabla \theta^\eps\|_{L^1([0,T],L^2)} + \eps^{1/2-d/8}\right) +
\eps^{1/2} \\
\label{5}
& \lesssim \| \eps \nabla \theta^\eps\|_{L^1([0,T],L^2)} + \eps^{1/2}.
\end{align}
Then, using Gronwall lemma, we obtain
$$\| \eps \nabla \theta^\eps (t)\|_{L^2} \lesssim \eps^{1/2}, \quad \forall t \in [0,T].$$
Gagliardo-Nirenberg inequality then implies
\begin{align*}
 \| \theta^\eps(t)\|_{L^4} & \lesssim \eps^{-d/4}\| \theta^\eps(t)\|_{L^2}^{1-d/4}\| \eps \nabla \theta^\eps(t)\|_{L^2}^{d/4}, \\
& \lesssim \eps^{-d/4+1/2} \lesssim \eps^{-d/8}\eps^{1/2-d/8},
\end{align*}
with $1/2-d/8 >0$ because $d \leq 3$. We infer that \eqref{bootstrap} holds for finite time. 
This concludes the bootstrap argument and we infer
$$\sup_{0 \leq t \leq T} 
\left(\|\theta^\eps(t)\|_{L^2} + \|\eps \nabla \theta^\eps(t)\|_{L^2}  \right) \Tend\eps 0 0.$$
Theorem \ref{main} then follows using Proposition \ref{L2} and the relation $\theta^\eps = w^\eps+ \eps g^\eps \chi_-$.
\begin{remark}
 \label{Perturbation2}
If the initial datum is such that :
$$\psi^\eps_0(x)= \eps^{-d/4} e^{i\xi_0^+.
    (x-x_0^+)/\eps}a
\left(\frac{x-x_0^+}{\sqrt\eps}\right) \chi_+(x) + \eta^\eps (x), $$
with $\eta^\eps(x)+ \eps \nabla \eta^\eps(x)  = \mathcal{O}(\eps^{\gamma_0})$ in $L^2$, 
the initial terms $\theta^\eps (0,x)$ presents, a $\mathcal{O}(\eps^{\gamma_0})$ contribution in $H^1_\eps$, 
and we add a term in the estimate \eqref{***}. Therefore, performing the same bootstrap argument \eqref{bootstrap}, 
we obtain for $t \in [0,T]$ :
\begin{equation*}
\| \theta^\eps \|_{L^p([0,t],L^q)}  \lesssim  \eps^{1/2 - d/8} + \eps^{-d/8+\gamma_0},\\
\end{equation*}
and this gives
\begin{equation*}
\| \theta^\eps \|_{L^\infty([0,t],L^2)}  \leq C(T)  \left(\eps^{1/2} + \eps^{\gamma_0}\right).\\
\end{equation*}
The estimate for the derivative writes :
\begin{equation*}
 \| \eps \nabla \theta^\eps \|_{L^p([0,t],L^q)}  \lesssim  \eps^{1/2 - d/8} + \eps^{-d/8+\gamma_0} 
+ \eps^{-d/8} \| \eps \nabla \theta^\eps \|_{L^1([0,t],L^2)},
\end{equation*}
and finally, by Gronwall Lemma, for $t \in [0,T]$ :
\begin{equation*}
 \| \eps \nabla \theta^\eps (t) \|_{L^2}  \leq  C(T)\left(\eps^{1/2} + \eps^{\gamma_0}\right).
\end{equation*}
Gagliardo-Nirenberg inequality then gives
\begin{align*}
 \| \theta^\eps (t)\|_{L^q} \lesssim & \; \eps^{-d/4}\|\theta^\eps (t) \|_{L^2}^{1-d/4}\; \| \eps \nabla \theta^\eps (t) \|_{L^2}^{d/4} \\
\lesssim & \; \eps^{-d/8}\eps^{1/2-d/8}+ \eps^{\gamma_0-d/8},
\end{align*}
with $1/2-d/8>0$ for $d \leq 3$ and $\gamma_0 > d/8$. This implies that the bootstrap argument \eqref{bootstrap} holds for finite time, 
whence the result of Remark \ref{perturbation}.
\end{remark}
Let us notice that the proof of Theorem \ref{main} crucially relies on the result of Proposition \ref{A}, for $k\leq 6$. 
Therefore, in order to deal with large times, we need to have the exponential control \eqref{B}.

\subsection{Large time case}
\label{largetimecase}
Our aim is now to prove Theorem \ref{loglog} and to find that the approximation holds until 
$T^\eps=C \log \log (\eps^{-1})$, for some suitable $C>0$. 
\\
\noindent
We assume that we have the exponential control \eqref{B} for $k \leq 6$, which gives us the following 
estimate on $\varphi^\eps(t)$:
$$\| \varphi^\eps (t)\|_{L^q(\R^d)} \lesssim \eps^{-d/8}e^{C't}. $$
Therefore, we make the following bootstrap assumption on $\theta^{\eps}(t)$:
$$\|\theta^\eps(t)\|_{L^q(\R^d)}\lesssim \eps^{-d/8}e^{C't}, \quad t \in [0,T], $$
where $C'$ denotes the same constant as above.
By Theorem \ref{main}, for any $T>0$ independent of $\eps$, the bootstrap assumption is satisfied, provided $\eps \in ]0;\eps_T]$. 
\\
\noindent
We recall the estimate of the proof of Theorem \ref{main}, with $I=[t, t+\tau]$, $t\geq 0, \tau>0$, $s \in I$ before the absorption argument :
\begin{align*}
\| \theta^\eps \|_{L^p(I, L^q)}  \lesssim & \; \eps^{-1/p}\| \theta^\eps (s)\|_{L^2} + \eps^{- 1/p} \| L^\eps \|_{L^1(I, L^2)} 
+ \eps^{d/2 - 2/p} \left(\|\varphi^\eps \|_{L^\sigma(I, L^q)}^2 \right. \\
& \left. + \|\theta^\eps \|_{L^\sigma(I, L^q)}^2  + \eps^2 \|g^\eps \|_{L^\sigma(I, L^q)}^2 \right)
\left( \|\theta^\eps \|_{L^p(I, L^q)} + \eps \|g^\eps \|_{L^p(I, L^q)} \right). 
\end{align*}
To simplify notations, we assume $\tau \leq 1$. Then we use the new estimates on $\varphi^\eps(t)$ and as long as the bootstrap argument 
holds, choosing a larger $C'$ if necessary, and by integration in $s$, we have :
\begin{align*}
\| \theta^\eps \|_{L^p(I, L^q)}& 
\leq K \left(\eps^{-1/p}\tau^{-1}\| \theta^\eps\|_{L^1(I,L^2)} + \eps^{- 1/p} \| L^\eps \|_{L^1(I, L^2)} \right. \\
& \quad + \left. \tau^{2/\sigma}e^{2C't}\| \theta^\eps \|_{L^p(I, L^q)}+ \eps^{1-d/4} \tau^{2/\sigma+1/p}e^{C't} \right),
\end{align*}
where $K$ is a constant independent of $\eps$. 
We want to apply the absorption argument, and to have the term $K \tau^{2/\sigma}e^{2C't} \| \theta^\eps \|_{L^p(I, L^q)}$ be absorbed 
by the left hand side. We first notice that for $t\leq A \log(\eps^{-1})$, 
$$K\tau^{2/\sigma}e^{2C't} \leq K\tau^{2/\sigma}\eps^{-2AC'}.$$
We choose $\tau >0$ such that
$$K \tau^{2/\sigma}\eps^{-2AC'} \leq \dfrac{1}{2}, $$
which implies that $\tau$ and $\tau^{-1}$ are bounded by constants independent of $t$. 
We then obtain
\begin{equation*}
\| \theta^\eps \|_{L^p(I, L^q)}
\lesssim \eps^{-1/p}\| \theta^\eps\|_{L^1(I,L^2)} + \eps^{- 1/p} \| L^\eps \|_{L^1(I, L^2)} 
+ \eps^{1-d/4}e^{C't}.
\end{equation*}
We recover $[0,t]$ with a finite number of intervals of the form $[j\tau,(j+1)\tau]$ and we obtain for $t \leq A \log (\eps^{-1})$ :
\begin{equation}
\label{a}
 \| \theta^\eps \|_{L^p([0,t], L^q)} \lesssim 
\eps^{-1/p}\| \theta^\eps\|_{L^1([0,t],L^2)} + \eps^{- 1/p} \| L^\eps \|_{L^1([0,t], L^2)}
 + \eps^{1-d/4} e^{C't}.
\end{equation}
Then, thanks to Strichartz estimates again, we have
\begin{align*}
 \| \theta^\eps \|_{L^\infty([0,t], L^2)} & \lesssim \| L^\eps\|_{L^1([0,t],L^2)} + \eps^{- 1/p}\| NL^\eps\|_{L^{p'}([0,t],L^{q'})},\\
& \lesssim \| L^\eps\|_{L^1([0,t],L^2)} 
+ \eps^{d/2 - 1/p} \left(\|\varphi^\eps \|_{L^\sigma([0,t], L^q)}^2 + \|\theta^\eps \|_{L^\sigma([0,t], L^q)}^2 \right. \\
& \quad \left.   + \eps^2 \|g^\eps \|_{L^\sigma([0,t], L^q)}^2 \right)
\left( \|\theta^\eps \|_{L^p([0,t], L^q)} + \eps \|g^\eps \|_{L^p([0,t], L^q)} \right) \\
& \lesssim \| L^\eps\|_{L^1([0,t],L^2)} + \eps^{1/p} t^{2/\sigma} e^{Ct} 
\left(\eps^{-1/p}\| \theta^\eps\|_{L^1([0,t],L^2)} \right. \\
& \quad \left. + \eps^{- 1/p} \| L^\eps \|_{L^1([0,t], L^2)}
 + \eps^{1-d/4} e^{Ct}  \right) \\
 & \lesssim \; \eps^{1/2}e^{Ct} + \| \theta^\eps\|_{L^1([0,t],L^2)} e^{Ct}.
\end{align*}
Finally, Gronwall lemma yields :
\begin{equation}
\label{b}
\| \theta^\eps (t) \|_{L^2} \lesssim \eps^{1/2}e^{e^{Ct}}, \quad \forall t \geq 0.
\end{equation}
From \eqref{a} and \eqref{b}, we infer the following estimate :
\begin{equation*}
 \| \theta^\eps \|_{L^p([0,t], L^q)} \lesssim \eps^{1/2-d/8} e^{e^{Ct}}.
\end{equation*}
We have to check how long the bootstrap argument holds; using the same method for $\eps \nabla \theta^\eps$, we obtain
\begin{align*}
 \| \eps \nabla \theta^\eps \|_{L^p(I, L^q)} & \lesssim \; \eps^{-1/p} \| \eps \nabla \theta^\eps (t)\|_{L^2} 
+ \eps^{-1/p} \| \eps \nabla V \theta^\eps \|_{L^1(I, L^2)} \\
& \quad + \eps^{-1/p} \| \eps \nabla L^\eps \|_{L^1(I, L^2)} + \eps^{-2/p} \| \eps \nabla NL^\eps \|_{L^{p'}(I, L^{q'})}.
\end{align*}
The nonlinear term writes
\\
$\eps^{-2/p} \| \eps \nabla NL^\eps \|_{L^{p'}(I, L^{q'})} $
\begin{align*}
& \lesssim \; \eps^{d/2 - 2/p} \left( \|\varphi^\eps \|_{L^\sigma(I,L^q)}^2 + \| \theta^\eps \|_{L^\sigma(I,L^q)}^2 + \eps^2 \|g^\eps \|_{L^\sigma(I,L^q)}^2 \right)\\
& \qquad \left( \| \eps \nabla \theta^\eps\|_{L^p(I, L^q)} + \eps^2 \| \nabla (g^\eps \chi_-)\|_{L^p(I, L^q)} \right) \\
& \qquad + \eps^{d/2 - 2/p}\| \varphi^\eps \; \eps\nabla \varphi^\eps\|_{L^{\sigma/2}(I,L^2)} 
\left(\| \theta^\eps\|_{L^p(I, L^q)} + \eps \|g^\eps\|_{L^p(I, L^q)} \right), \\
& \lesssim  \tau^{2/\sigma} e^{Ct} \| \eps \nabla \theta^\eps\|_{L^p(I, L^q)} + \tau^{2/\sigma} \eps^{1/2-d/8} e^{e^{Ct}},
\end{align*}
and we have for all $s \in I$ :
\begin{align*}
 \| \eps \nabla \theta^\eps \|_{L^p(I, L^q)}  \\
\lesssim & \;\eps^{-1/p} \| \eps \nabla \theta^\eps (s)\|_{L^2} 
+ \eps^{-1/p} \| \eps \nabla V \theta^\eps \|_{L^1(I, L^2)}+ \eps^{-1/p} \| \eps \nabla L^\eps \|_{L^1(I, L^2)} \\
& \quad  + \tau^{2/\sigma} e^{2Ct} \| \eps \nabla \theta^\eps\|_{L^p(I, L^q)} 
+ \tau^{2/\sigma} \eps^{1/2-d/8} e^{e^{Ct}} \\
\leq & K \left( \eps^{-1/p} \| \eps \nabla \theta^\eps (s)\|_{L^2} 
+ \tau^{2/\sigma} e^{Ct} \| \eps \nabla \theta^\eps\|_{L^p(I, L^q)}+ \tau^{2/\sigma} \eps^{1/2-d/8} e^{e^{Ct}}\right)
\end{align*}
Integrating on $I$, we have
\begin{multline*}
 \| \eps \nabla \theta^\eps \|_{L^p(I, L^q)} \leq  K \left( \eps^{-1/p} \tau^{-1}\| \eps \nabla \theta^\eps\|_{L^1(I,L^2)}\right. \\
\left. + \tau^{2/\sigma} e^{Ct} \| \eps \nabla \theta^\eps\|_{L^p(I, L^q)}+ \tau^{2/\sigma} \eps^{1/2-d/8} e^{e^{Ct}}\right).
\end{multline*}
With $t\leq A \log(\eps^{-1})$, we choose $\tau$ such that 
$$K\tau^{2/\sigma} \eps^{-2AC}\leq \dfrac{1}{2},$$
and repeating the same procedure, we obtain
\begin{align*}
 \| \eps \nabla \theta^\eps \|_{L^p([0,t], L^q)} &\lesssim \eps^{-1/p} \| \eps \nabla \theta^\eps\|_{L^1([0,t],L^2)} 
+ \eps^{1/2-d/8} e^{e^{Ct}}.
\end{align*}
Then, using Strichartz estimates again, we find
\begin{equation*}
 \| \eps \nabla \theta^\eps \|_{L^\infty([0,t], L^2)} \lesssim  \;e^{Ct}\| \eps \nabla \theta^\eps \|_{L^1([0,t], L^2)} +
\eps^{1/2} e^{e^{Ct}},
\end{equation*}
and Gronwall lemma yields
\begin{equation*}
 \| \eps \nabla \theta^\eps (t) \|_{L^2} \lesssim \eps^{1/2} e^{e^{Ct}}, \quad \forall t \geq 0.
\end{equation*}
It remains to check that the bootstrap argument holds for $t \leq c\log \log(\eps^{-1})$, for $c$ sufficiently small. 
We use the weighted Gagliardo-Nirenberg inequality to have :
\begin{align*}
\| \theta^\eps (t) \|_{L^q} & \lesssim \eps^{-d/4} \| \theta^\eps (t) \|_{L^2}^{1-d/4} \| \eps \nabla \theta^\eps (t) \|_{L^2}^{d/4}, \\
& \leq \widetilde{K} \eps^{-d/4} \eps^{1/2}e^{e^{\tilde{C}t}}.
\end{align*}
Therefore, taking $\eps$ sufficiently small, the bootstrap argument holds as long as
$$\widetilde{K} \eps^{1/2}\eps^{-d/8} e^{e^{\tilde{C}t}} \leq e^{C't}, $$
with $1/2 - d/8 >0$. We check that for large $t$ and $\eps$ sufficiently small, it remains true for 
$t \leq c \log \log(\eps^{-1})$, 
with $c$ independent of $\eps$. The proof of Theorem \ref{loglog} is now complete.
\begin{remark}
In order to deal with initial data which are perturbation of wave packets, as in Remark \ref{perturbation}, but in large time case, 
we have to check how long the bootstrap argument holds with new terms appearing from $\eta^\eps$. We have (using the estimates 
obtained in Remark \ref{Perturbation2}) :
\begin{align*}
 \| \theta^\eps (t) \|_{L^2} \lesssim \left(\eps^{1/2}+ \eps^{\gamma_0}\right)e^{e^{Ct}} \\
\| \eps \nabla \theta^\eps (t) \|_{L^2} \lesssim \left(\eps^{1/2}+ \eps^{\gamma_0}\right)e^{e^{Ct}},
\end{align*}
for all $t \geq0$, with $\gamma_0 > d/8$. Using Gagliardo-Nirenberg inequality, we find
\begin{align*}
 \| \theta^\eps (t) \|_{L^4} \lesssim & \; \eps^{-d/4} \| \theta^\eps (t) \|_{L^2}^{1-d/4} \| \eps \nabla \theta^\eps (t) \|_{L^2}^{d/4} \\
\lesssim & \; \left( \eps^{1/2-d/8}+\eps^{\gamma_0-d/8}\right) \eps^{-d/8}e^{e^{Ct}}.
\end{align*}
The bootstrap argument holds as long as
$$ \left( \eps^{1/2-d/8}+\eps^{\gamma_0-d/8}\right) \eps^{-d/8}e^{e^{Ct}} \ll \eps^{-d/8}e^{Ct},$$
and since $\gamma_0 >d/8$, the above condition is ensured for $t \leq C \log \log (\eps^{-1})$ for some suitable $C$, 
and this gives the approximation for large times.
\end{remark}

\section{Growth of Sobolev norms and momenta of the profile}
\label{sec:growth}
In this section, we will focus on the behaviour of $u(t)$ for large time and prove Proposition \ref{growth}, which gives an additional case 
where Theorem \ref{loglog} holds.
\smallbreak
\noindent
We first recall some results that follow from \cite{Fujiwara} (see Theorem 2.2 and Corollary 2.7 of this paper) 
and \cite{KT}. 
We consider $Q^+$ satisfying 
$$\sup_{t \in \R} |Q^+(t)| < +\infty,$$ 
then, the following (local in time) Strichartz estimates hold.
\begin{proposition}[Strichartz estimates for quadratic potentials]
 \label{StrichartzProfil}
Let $(p,q)$, $(p_1,q_1)$ be admissible pairs, defined in Definition \ref{def:adm}. 
Let $I$ be a finite time interval. We consider $u$, the solution to
\begin{equation*}
  i \d_t u + \dfrac{1}{2} \Delta u - \dfrac{1}{2} \left\langle Q^+(t)x;x \right\rangle u = f \quad ; \quad u(0,x)=u_0(x).
\end{equation*}
where $u_0 \in L^2(\R^d)$ and $f \in L^{p_1'}(I,L^{q_1'})$.
Then, there exists $C=C(q, q_1,|I|)$, such that for all $s \in I$ 
$$\|u\|_{L^p(I,L^q)} \leq C \|u_0 \|_{L^2} + \|f\|_{L^{p_1'}(I,L^{q_1'})}=C \|u(s) \|_{L^2} + \|f\|_{L^{p_1'}(I,L^{q_1'})}.$$
\end{proposition}
\noindent
We prove the following lemma, since it is the first step of the proof of Proposition~\ref{growth}.
\begin{lemma}
 \label{Q+}
Let $d=2$ or $3$. Assume $\Lambda \geq 0$ and :
\begin{equation*}
 \left|\dfrac{d}{dt}Q^+(t)\right| \leq \dfrac{C}{(1+|t|)^{\kappa_0+1}},
\end{equation*}
with $\kappa_0 >2$. We consider $u$, the solution to the Cauchy problem \eqref{profil}. 
Then, there exist $C_1, C_2>0$ and $\gamma >0$ such that
\begin{equation*}
 \|\nabla u(t) \|_{L^2} \leq C_1, \quad \|xu(t) \|_{L^2} \leq C_2 (1+|t|)^{\gamma+1}.
\end{equation*}
Moreover 
$$\|\nabla u\|_{L^p([0,t],L^q)} + \|x u\|_{L^p([0,t],L^q)} \lesssim (1+|t|)^{\gamma+2}. $$
\end{lemma}
\begin{proof}
We use an energy argument and set 
$$E(t) = \dfrac{1}{2} \| \nabla u(t)\|_{L^2}^2 + \dfrac{\Lambda}{4} \| u(t) \|_{L^4}^4 + \dfrac{1}{2} \int_{\R^d} 
\left\langle Q^+(t)x,x \right\rangle |u(t,x)|^2 dx, $$
and
$$V(t) = \dfrac{1}{2} \int_{\R^d} |x|^2 |u(t,x)|^2 dx.$$
We have :
$$E'(t) = \dfrac{1}{2} \int_{\R^d} \left\langle \dfrac{d}{dt}Q^+(t)x;x \right\rangle |u(t,x)|^2 dx, $$
$$V'(t) = 2 \; Im \; \int_{\R^d} \left(x.\nabla u(t,x) \right)\overline{u}(t,x) dx.$$
We introduce the following quantity :
$$A(t) = E(t) + \dfrac{\widetilde{C}}{(1+|t|)^{2+\delta}} V(t), $$
with $\delta>0$. For $\widetilde{C}$ large enough and choosing $\delta \ll 1$ such that $2 <2+\delta < \kappa_0$, we obtain
$$A(t) \geq  \dfrac{1}{2} \| \nabla u(t)\|_{L^2}^2 + \dfrac{1}{4} \| u(t) \|_{L^4}^4 + \dfrac{1}{(1+|t|)^{2+\delta}} V(t).$$
Then, using all these estimates, we have :
\begin{equation*}
E'(t) \lesssim \dfrac{1}{(1+|t|)^{\kappa_0+1}} V(t)^{1/2} \quad ; \quad V(t)  \leq \left( 1+ |t|\right)^{-2-\delta} A(t),
\end{equation*}
and combining these estimates, we obtain
\begin{align*}
 A'(t) & = E'(t) + \dfrac{\widetilde{C}}{(1+|t|)^{2+\delta}}V'(t) - \dfrac{\widetilde{C} (2+\delta)}{(1+|t|)^{3+\delta}} V(t), \\
& \leq E'(t) + \dfrac{\widetilde{C}}{(1+|t|)^{2+\delta}}V'(t), \\
& \leq \left(\dfrac{C}{(1+|t|)^{\kappa_0-1-\delta}} + \dfrac{C}{(1+|t|)^{1+\delta/2}}\right) A(t), \\
& \leq \dfrac{C}{(1+|t|)^{1+\widetilde{\delta}}} A(t),
\end{align*}
for some $\widetilde{\delta} >0$. By Gronwall Lemma, we find $|A(t)| \leq C$, whence
\begin{equation}
 \label{H1}
\|\nabla u(t) \|_{L^2} \leq C, \quad \|x u(t) \|_{L^2} \leq C (1+|t|)^{1+\widetilde{\gamma}},
\end{equation}
where $\widetilde{\gamma}>0$. And since $\| u(t) \|_{L^4}^4 \leq A(t)$, we deduce
$$\|u(t)\|_{L^4} \leq C. $$
\\
Besides, the derivative and the first momentum of $u$ satisfy the following equation:
\begin{equation*}
  i \d_t (\nabla u) + \dfrac{1}{2} \Delta (\nabla u) - \dfrac{1}{2} \left\langle Q^+(t)x;x \right\rangle (\nabla u) =
Q^+(t)x \; u +\Lambda \nabla \left(|u|^{2}u \right),
\end{equation*}
and
\begin{equation*}
 i \d_t (xu) + \dfrac{1}{2} \Delta (xu) - \dfrac{1}{2} \left\langle Q^+(t)x;x \right\rangle (xu) =
\nabla u + \Lambda |u|^{2}\left( xu \right).
\end{equation*}
We set $I=[t,t+\tau]$, with $t\geq 0$, $\tau >0$, and $(p,q)$ any admissible pair. 
We recall that $(8/d,4)$ is the admissible pair introduced in the previous proofs, and $\sigma = 8/(4-d)$, 
which will be useful for the absorbtion argument. 
We have, thanks to Strichartz estimates given in Proposition \ref{StrichartzProfil} :
\\
$\quad\| \nabla u\|_{\left(L^p(I,L^q)\right)\cap\left(L^{8/d}(I,L^4) \right)} $
\begin{align*}
\lesssim & \; \| \nabla u(t) \|_{L^2} +
\| Q^+(t)xu\|_{L^1(I,L^2)}+ \|\nabla \left(|u|^2 u\right)\|_{L^{8/(8-d)}(I,L^{4/3})} \\
\lesssim & \; \| \nabla u(t) \|_{L^2} + \|xu\|_{L^1(I,L^2)} + \| u\|_{L^\sigma(I,L^4)}^2 \|\nabla u\|_{L^{8/d}(I,L^4)} \\
\lesssim & \;  \| \nabla u(t) \|_{L^2} + \|xu\|_{L^1(I,L^2)} + \tau^{2/\sigma}\| u\|_{L^\infty(I,H^1)}^2 \|\nabla u\|_{L^{8/d}(I,L^4)},
\end{align*}
where we have used H\"older inequality and the Sobolev embedding (with $d \leq 4$). Then we have by \eqref{H1}
\begin{equation}
\label{u'}
\| \nabla u\|_{\left(L^p(I,L^q)\right)\cap\left(L^{8/d}(I,L^4) \right)} 
\lesssim \| \nabla u(t) \|_{L^2} + \|xu\|_{L^1(I,L^2)} + \tau^{2/\sigma}\|\nabla u\|_{L^{8/d}(I,L^4)}.
\end{equation}
Similarly, we have :
\begin{align}
\notag
  \| x u\|_{\left(L^p(I,L^q)\right)\cap\left(L^{8/d}(I,L^4) \right)}   \lesssim & \| xu(t) \|_{L^2} +
\| \nabla u\|_{L^1(I,L^2)}+ \||u|^2 \left(xu\right)\|_{L^{8/(8-d)}(I,L^{4/3})} \\
\label{xu}
 \lesssim    \| xu(t) \|_{L^2}  & + \| \nabla u\|_{L^1(I,L^2)} + \tau^{2/\sigma}\| u\|_{L^\infty(I,H^1)}^2 \|x u\|_{L^{8/d}(I,L^4)}.
\end{align}
Combining \eqref{u'} and \eqref{xu}, and choosing $\tau$ sufficiently small, we obtain
\\
$ \quad \| \nabla u\|_{\left(L^p(I,L^q)\right)\cap\left(L^{8/d}(I,L^4) \right)} + 
 \| x u\|_{\left(L^p(I,L^q)\right)\cap\left(L^{8/d}(I,L^4) \right)}$
\begin{align*}
\qquad \qquad &\lesssim \; \| \nabla u(t) \|_{L^2} + \| xu(t) \|_{L^2} + \| \nabla u\|_{L^1(I,L^2)} + \|xu\|_{L^1(I,L^2)} \\
\qquad \qquad &\lesssim \; (1+(t+\tau)^{\gamma+2}),
\end{align*}
with $\gamma>0$, where we have used \eqref{H1}. 
\\
\noindent
Using the above estimate for $t = 0, \tau, 2\tau, \cdots , j \tau$ for $j \in \N$, and
by induction on $j$, we finally obtain the following estimate on $[0,t]$, for all $t \geq 0$ :
\begin{equation}
\label{H1bis}
\| \nabla u\|_{\left(L^p([0,t],L^q)\right)\cap\left(L^{8/d}([0,t],L^4) \right)} 
 +  \| x u\|_{\left(L^p([0,t],L^q)\right)\cap\left(L^{8/d}([0,t],L^4) \right)}
 \lesssim (1+t)^{\gamma+2},
\end{equation}
for any admissible pair; and the proof of the lemma is complete.
\end{proof}
\noindent
We now prove Proposition \ref{growth}.
\begin{proof}[Proof of Proposition \ref{growth}]
We argue by induction. The case $k=1$ is given by Lemma \ref{Q+}. 
Assume now that the result holds for $k-1$; we will prove it for $k$. 
The first important point is to notice that it suffices to analyse the derivatives and momenta of order $k$. 
In fact, the following inequality holds :
\begin{align*}
 \sum_{|\alpha|+|\beta| \leq k} \|x^\alpha \d_x^\beta u \|_{L^q} \leq C 
\left[ \|(1+|x|)^{k}u \|_{L^q} + \sum_{|\alpha|\leq k} \|\d_x^\alpha u\|_{L^q} \right].
\end{align*}
It is an easy consequence of Theorem 5 of \cite{Triebel}.
\\
\noindent
For all $\alpha \in \N^d$ such that $|\alpha| = k$, $u$ satisfies the following equations :
\begin{multline*}
 i \d_t (\d_x^\alpha u) + \dfrac{1}{2} \Delta (\d_x^\alpha u) - \dfrac{1}{2} \left\langle Q^+(t)x;x \right\rangle (\d_x^\alpha u) = \\
\dfrac{1}{2} \left[ \d_x^\alpha, \left\langle Q^+(t)x;x \right\rangle  \right]u  +\Lambda \d_x^\alpha \left(|u|^{2}u \right),
\end{multline*}
and 
\begin{equation*}
 i \d_t (x^\alpha u) + \dfrac{1}{2} \Delta (x^\alpha u) - \dfrac{1}{2} \left\langle Q^+(t)x;x \right\rangle (x^\alpha u) =
\dfrac{1}{2} \left[ \Delta, x^\alpha \right] u + \Lambda |u|^{2}\left( x^\alpha u \right).
\end{equation*}
For conveniance, we distinguish cases $d=2$ and $d=3$.
\smallbreak
\noindent
\textbf{Case $d=2$ : } Here, the usual admissible pair, used to deal with the nonlinearity, is $(8/d,4)=(4,4)$, and $\sigma = 8/(4-d)=4$. 
Strichartz estimates on $I=[t, t+\tau]$, for $t\geq 0$ and $\tau>0$ then yield :
\begin{align*}
 \| \d_x^\alpha u\|_{\left(L^\infty(I,L^2)\right)\cap\left(L^4(I,L^4)\right)}  \lesssim & \; \| \d_x^\alpha u(t) \|_{L^2} 
+ \| \d_x^\alpha (|u|^2u) \|_{L^{4/3}(I,L^{4/3})} \\
& \quad + \|\left[ \d_x^\alpha, \left\langle Q^+(t)x;x \right\rangle  \right]u \|_{L^1(I,L^2)}, \\
\| x^\alpha u\|_{\left(L^\infty(I,L^2)\right)\cap\left(L^4(I,L^4)\right)}  \lesssim & \| x^\alpha u(t) \|_{L^2} + 
\| |u|^2(x^\alpha u) \|_{L^{4/3}(I,L^{4/3})} \\
& \quad + \|\left[ \Delta, x^\alpha \right]u \|_{L^1(I,L^2)}.
\end{align*}
We write 
$$\left[ \d_x^\alpha, \left\langle Q^+(t)x;x \right\rangle  \right]u = \sum_{|\beta| = |\alpha|-1} \left(c_\beta (t)x \right) \d_x^\beta u
+ \sum_{|\gamma|=|\alpha|-2} d_\gamma(t) \d_x^\gamma u ,$$
where $c_\beta$ and $d_\gamma$ are bounded for all $t \in \R$. 
We first notice that the derivative of the nonlinearity satisfy :
\begin{equation*}
 \left| \d_x^\alpha (|u|^2u) \right| \lesssim |u|^2 |\d_x^\alpha u| + \sum_{j \leq J} \left| w_{j1}\right|\; \left| w_{j2}\right| \; 
\left| w_{j3}\right|,
\end{equation*}
where $J \in \N$, $w_{jl}$ are derivatives of $u$ or $\overline{u}$ of order lower than $k-1$, rearranged such that 
such that $w_{j1}$ is of order lower than $w_{j2}$, which is of order lower than $w_{j3}$. 
Then, thanks to H\"older inequality, and Sobolev embedding, we have
\begin{align*}
 \| \d_x^\alpha u\|_{\left(L^\infty(I,L^2)\right)\cap\left(L^4(I,L^4)\right)}  \lesssim & \; \| \d_x^\alpha u(t) \|_{L^2}
+ \tau^{2/\sigma} \|u\|_{L^{\infty}(I,H^1)}^2 \| \d_x^\alpha u \|_{L^{4}(I,L^{4})} \\
& \quad + \sum_{j \leq J}\| w_{j1}\|_{L^{4}(I,L^{4})} \|w_{j2} \|_{L^{4}(I,L^{4})} \|w_{j3}\|_{L^{4}(I,L^{4})} \\
& \quad + \sum_{|\beta|=|\alpha|-1} \|x \d_x^\beta u\|_{L^1(I,L^2)}+ \tau \mathcal{F}_{k-1}, \\
\| x^\alpha u\|_{\left(L^\infty(I,L^2)\right)\cap\left(L^p(I,L^q)\right)}  \lesssim & \; \| x^\alpha u(t) \|_{L^2} + 
\tau^{2/\sigma} \|u\|_{L^{\infty}(I,H^1)}^2 \|x^\alpha u \|_{L^{4}(I,L^{4})} + \tau \mathcal{F}_{k-1},
\end{align*}
where $\mathcal{F}_{k-1}$ is the sum of $L^\infty(I,L^2)-$norm of terms of order lower than $k-1$ 
(it contains terms from $\sum_{|\gamma|=|\alpha|-2} d_\gamma(t) \; \d_x^\gamma u$ and $\left[ \Delta, x^\alpha \right]u$ ). 
We have, thanks to the induction hypothesis
\begin{align*}
\sum_{j \leq J}\| w_{j1}\|_{L^{4}(I,L^{4})} \|w_{j2}\|_{L^{4}(I,L^{4})} \| w_{j3}\|_{L^{4}(I,L^{4})} \lesssim e^{Ct}.
\end{align*}
Choosing $\tau \ll 1$, sufficiently small, the nonlinear term can be absorbed by the left handside term, 
and using \eqref{H1} and the induction hypothesis :
\begin{align}
\label{exp1}
 \| \d_x^\alpha u\|_{\left(L^\infty(I,L^2)\right)\cap\left(L^4(I,L^4)\right)}  \lesssim & \; \| \d_x^\alpha u(t) \|_{L^2}
+ \sum_{|\beta|=|\alpha|-1} \|x \d_x^\beta u\|_{L^1(I,L^2)} + e^{C(t+\tau)}, \\
\notag
\| x^\alpha u\|_{\left(L^\infty(I,L^2)\right)\cap\left(L^4(I,L^4)\right)}  \lesssim & \; \| x^\alpha u(t) \|_{L^2}
+  e^{C(t+\tau)}.
\end{align}
Writing $$A_k(t) = \max_{0\leq s\leq t} \sum_{|\alpha|+|\beta| \leq k} \|x^{\alpha} \d_x^{\beta}u(t)\|_{L^2},$$
we obtain the following inequality :
$$A_k(t+\tau) \leq C \; A_k(t) + Ce^{C(t+\tau)},$$
with $\tau \ll 1$.
Set $u_n = A_k(n\tau)$, for $n \in \N$. Then : $u_{n+1} \leq Cu_n + Ce^{C(n+1)\tau}$. With $v_n = e^{-Kn\tau} u_n$, this inequality
allows us to find, by induction on $n$, that $v_n$ is bounded for all $n \in \N$. We deduce that $u_n$ grows exponentially, and that 
$A_k$ grows in the same way. 
\\
\noindent
To prove the property for any admissible pair $(p,q)$, we go back to \eqref{exp1} and use the previous estimate, found for $(\infty,2)$:
\begin{align*}
 \| \d_x^\alpha u\|_{\left(L^p(I,L^q)\right)\cap\left(L^4(I,L^4)\right)}  \lesssim & \; \sup_{s \in I}\| \d_x^\alpha u(s) \|_{L^2}
+ \sum_{|\beta|=|\alpha|-1} \|x \d_x^\beta u\|_{L^1(I,L^2)} + e^{C(t+\tau)}, \\
\lesssim & \; e^{C(t+\tau)}\\
\| x^\alpha u\|_{\left(L^p(I,L^q)\right)\cap\left(L^4(I,L^4)\right)}  \lesssim & \; \sup_{s \in I}| x^\alpha u(s) \|_{L^2}
+  e^{C(t+\tau)}\\
 \lesssim & \; e^{C(t+\tau)},
\end{align*}
and we recover $[0,t]$ to obtain the property for the case $d=2$.
\smallbreak
\noindent
\textbf{Case $d=3$ : }The main difference comes from the products of
 derivatives of $u$. Since $d = 3$, we have $\sigma >p$ 
and the H\"older inequality used above fails in this case. 
Therefore, to deal with these terms, we choose a different admissible pair: $(2,6)$. It is important to notice that Strichartz estimates are 
available for this endpoint (see \cite{KT} for details). The notations will be the same as for case $d=2$.
\\
\noindent
Strichartz estimates on $I=[t, t+\tau]$, for $t\geq 0$ and $\tau>0$ yield :
\begin{multline*}
 \| \d_x^\alpha u\|_{\left(L^\infty(I,L^2)\right)\cap\left(L^{8/3}(I,L^4)\right)}  \lesssim  \; \| \d_x^\alpha u(t) \|_{L^2} 
+ \| |u|^2 \d_x^\alpha u \|_{L^{8/5}(I,L^{4/3})}\\
+ \|\left[ \d_x^\alpha, \left\langle Q^+(t)x;x \right\rangle  \right]u \|_{L^1(I,L^2)} 
+ \sum_{j \leq J}\left\| w_{j1}\; w_{j2} \; w_{j3}\right\|_{L^{2}(I,L^{6/5})},
\end{multline*}
\begin{multline*}
\| x^\alpha u\|_{\left(L^\infty(I,L^2)\right)\cap\left(L^{8/3}(I,L^4)\right)}  \lesssim  \| x^\alpha u(t) \|_{L^2} + 
\| |u|^2(x^\alpha u) \|_{L^{8/5}(I,L^{4/3})} \\ + \|\left[ \Delta, x^\alpha \right]u \|_{L^1(I,L^2)}.
\end{multline*}
We recall that 
$$\left[ \d_x^\alpha, \left\langle Q^+(t)x;x \right\rangle  \right]u = \sum_{|\beta| = |\alpha|-1} \left(c_\beta (t).x \right) \d_x^\beta u
+ \sum_{|\gamma|=|\alpha|-2} d_\gamma(t) \d_x^\gamma u ,$$
where $c_\beta$ and $d_\gamma$ are bounded for all $t \in \R$. Then, we can write, thanks to H\"older inequality :
\begin{align*}
 \left\| w_{j1}\; w_{j2} \; w_{j3}\right\|_{L^2(I,L^{6/5})} & \lesssim 
\| w_{j1} \|_{L^\infty(I,L^6)} \; \| w_{j2} \|_{L^\infty(I,L^2)} \; \| w_{j3} \|_{L^2(I,L^6)} \\
& \lesssim \| w_{j1} \|_{L^\infty(I,H^1)}\; \| w_{j2} \|_{L^\infty(I,L^2)}\; \| w_{j3} \|_{L^2(I,L^6)} \\
& \lesssim e^{C(t+\tau)},
\end{align*}
where we have used the Sobolev embedding $\left(H^1 \hookrightarrow L^6 \right)$, and the induction hypothesis. 
We now infer, with the same arguments, assuming $\tau \leq 1$ :
\begin{align*}
 \| \d_x^\alpha u\|_{\left(L^\infty(I,L^2)\right)\cap\left(L^{8/3}(I,L^4)\right)}  \lesssim & \; \| \d_x^\alpha u(t) \|_{L^2} 
+\| u \|_{L^{8}(I,L^{4})}^2 \|\d_x^\alpha u \|_{L^{8/3}(I,L^{4})} \\
& \quad + \sum_{|\beta|=|\alpha|-1} \|x \d_x^\beta u\|_{L^1(I,L^2)}+ \tau \mathcal{F}_{k-1} 
+ e^{C(t+\tau)} , \\
\lesssim & \| \d_x^\alpha u(t) \|_{L^2} + \tau^{2/\sigma}\|\d_x^\alpha u \|_{L^{8/3}(I,L^{4})} \\
& \quad + e^{C(t+\tau)}, \\
\| x^\alpha u\|_{\left(L^\infty(I,L^2)\right)\cap\left(L^{8/3}(I,L^4)\right)}  \lesssim & \| x^\alpha u(t) \|_{L^2} + 
\| u \|_{L^{8}(I,L^{4})}^2 \|x^\alpha u \|_{L^{8/3}(I,L^{4})}+\tau \mathcal{F}_{k-1},\\
\lesssim &  \| x^\alpha u(t) \|_{L^2} + \tau^{2/\sigma}\|x^\alpha u \|_{L^{8/3}(I,L^{4})} + e^{C(t+\tau)}.
\end{align*}
We choose $\tau$ sufficiently small to absorb the nonlinear term and finally obtain
\begin{align*}
 \| \d_x^\alpha u\|_{\left(L^\infty(I,L^2)\right)\cap\left(L^{8/3}(I,L^4)\right)} 
\lesssim & \| \d_x^\alpha u(t) \|_{L^2} + e^{C(t+\tau)}, \\
\| x^\alpha u\|_{\left(L^\infty(I,L^2)\right)\cap\left(L^{8/3}(I,L^4)\right)} 
\lesssim &  \| x^\alpha u(t) \|_{L^2} + e^{C(t+\tau)}.
\end{align*}
Then, we use the same procedure as in the case $d=2$ and the proof is complete for $d=3$.
\end{proof}
Let us now sketch the proof of the property enunciated in Remark \ref{rmq:growth}. 
We first assume that 
$$E_0 > \lambda_\infty \quad \textrm{and that} \quad|x^+(t)| \Tend t \infty \infty.$$ 
Note that the eigenvalue has the same decreasing rate than the potential $V$, given by the long range property.
\\
\noindent
On one hand, we have $$\dfrac{d^2}{dt^2}|x^+(t)|^2 = 2 \left(|\dot{x}^+(t)|^2 - 2x^+(t).\nabla \lambda_+(x^+(t)) \right), $$
which gives
$$\lim\limits_{t \rightarrow \infty} \dfrac{d^2}{dt^2}|x^+(t)|^2 = 4 (E_0 - \lambda_\infty)>0, $$
thanks to the assumptions. We infer that for $t$ sufficiently large, $x^+(t)$ satisfies :
$$\dfrac{d}{dt}|x^+(t)|^2 \geq ct, \quad \textrm{for a small positive constant}\; c. $$ 
We finally obtain $$|x^+(t)|^2 \geq ct^2.$$
On the other hand, the derivative of $Q^+$ is given by
$$\dfrac{d}{dt}Q^+(t) =\dot{x}^+(t).\nabla\left({\rm Hess} \, \, \lambda_+ (x^+(t))\right). $$ 
We deduce from the conservation of the energy 
(and \eqref{trajestimate}) that $|\dot{x}^+(t)|$ is bounded. Besides, we have
$$\nabla\left({\rm Hess} \, \, \lambda_+ (x^+(t))\right) \leq c' \left\langle x^+(t) \right\rangle^{-p-3}. $$
Combining both previous estimates, we finally obtain the property \eqref{assumptionQ}.
\smallbreak
\noindent
Let us remark that the assumption
$$\lim\limits_{t \rightarrow \infty} |x^+(t)| = + \infty,$$ 
is not sufficient to prove that $Q^+$ satisfies \eqref{assumptionQ}, the assumption on the energy is essential. But if $E_0$ is such that 
$$E_0 > \lambda_\infty + \dfrac{1}{2} \sup_{x\in \R^d} (x.\nabla \lambda_+(x)),$$
then $\lim\limits_{t \rightarrow \infty} |x^+(t)| = + \infty$ and $Q^+$ satisfies \eqref{assumptionQ}.

\section{About the one dimensional case}
\label{d1}
\noindent
In this section, we assume $d=1$; we will prove Theorem \ref{1D}. 
\\
\noindent
We first notice that we are in the $L^2-$subcritical case. We consider a $2 \times 2$ system. 
The authors of \cite{Adiab} consider a matrix-valued potential which is \textit{at most quadratic} : $\rho$ and $\rho_0$ 
are at most quadratic and $\om$ is bounded as well as its derivatives. With these assumptions on the potential, the approximation is 
verified, up to a time $t^\eps = C \log \log(\eps^{-1})$. Under our assumptions on the potential $V$, 
it is possible to improve on this time $t^\eps$. 
Let us first notice that, because of the absence of crossing points, thanks to the resolvent estimates of \cite{Jecko3}, \cite{Jecko2}, 
we obtain Strichartz estimates, similar to the one in Theorem \ref{StrichartzEstimates} in the case $d=1$, 
following the same steps as in \cite{FerRou}. In fact, in \cite{FerRou}, the authors consider a matrix-valued potential 
with crossing points, and they assume $d = 2,3$ to avoid difficulties brought by these crossing points.
\\
\noindent
To obtain the approximation for large time, we follow the same steps as in the proof of Theorem \ref{loglog}, in Section \ref{largetimecase}.
The difference comes from the estimate on $\varphi^\eps$ and the bootstrap argument. We recall the estimate on $u$, 
proved in \cite{Ca-p} on  $I=[t, t+\tau]$, for the Lebesgue exponents $p=8, \; q=4$ and $\sigma = 8/3$:
\begin{equation*}
 \|u\|_{L^8(I,L^4)} \leq K \|u_0\|_{L^2} \lesssim 1.
\end{equation*}
This gives for $\varphi^\eps$, where $t \in I$ :
\begin{equation*}
 \| \varphi^\eps(t) \|_{L^4} = \eps^{-1/8} \| u(t) \|_{L^4},
\end{equation*}
and then, thanks to $p>\sigma$, we have 
\begin{align*}
\| \varphi^\eps \|_{L^{8/3}(I,L^4)} & \lesssim \tau^{1/4} \| \varphi^\eps \|_{L^8(I,L^4)}  \\
& \lesssim \tau^{3/8} \| \varphi^\eps \|_{L^{\infty}(I,L^4)} \\
& \lesssim \eps^{-1/8} \tau^{3/8} \| u \|_{L^\infty(I,L^4)} \lesssim \eps^{-1/8}.
\end{align*}
Consider $t\leq t^\eps$, with $t^\eps = A \log(\eps^{-1})$ where $A$ will be adjusted at the end. We perform a bootstrap argument. Assume :
\begin{equation}
 \label{bootstrap bis}
\|\theta^\eps (t)\|_{L^4} \lesssim \eps^{-1/8}.
\end{equation}
Then,
\begin{align*}
\| \theta^\eps \|_{L^8(I, L^4)}  \leq & K \left( \eps^{-1/8}\tau^{-1
} \| \theta^\eps\|_{L^1(I,L^2)} + \eps^{- 1/8} \| L^\eps \|_{L^1(I, L^2)}\right. \\
& \; \left.+ \eps^{1/4} \tau^{2/\sigma}\left(\eps^{-1/4}  + \eps^{3/2}e^{Ct} \right)
\left( \|\theta^\eps \|_{L^8(I, L^4)} + \eps^{3/4}e^{Ct} \right)\right).  
\end{align*}
We choose $A$ such that $\eps^{7/4}e^{Ct} <1$ for $t\leq A \log(\eps^{-1})$, (which gives $A< 7/(4C)$), 
and then, $\tau>0 $ such that $$2K \tau^{2/\sigma} \leq \dfrac{1}{2};$$
$\tau$ and $\tau^{-1}$ are both bounded by a constant independent of $t$, which gives 
\begin{equation}
\label{theta1}
 \| \theta^\eps \|_{L^8(I, L^4)}  \lesssim \eps^{-1/8} \| \theta^\eps\|_{L^1(I,L^2)} + \eps^{- 1/8} \| L^\eps \|_{L^1(I, L^2)} 
+\eps^{3/4}e^{Ct}.
\end{equation}
Using Strichartz estimates and H\"older inequality again, we have, for $t \geq 0$, $\tau >0$, fixed small enough :
\begin{align*}
\| \theta^\eps \|_{L^\infty(I, L^2)}  \lesssim & \; \| \theta^\eps(t) \|_{L^2} +\| L^\eps \|_{L^1(I, L^2)} + \eps^{3/8} 
\left(\|\varphi^\eps \|_{L^{8/3}(I, L^4)}^2 + \|\theta^\eps \|_{L^{8/3}(I, L^4)}^2 \right. \\
& \; \; \left.   + \eps^2 \|g^\eps \|_{L^{8/3}(I, L^4)}^2 \right)
\left( \|\theta^\eps \|_{L^8(I, L^4)} + \eps \|g^\eps \|_{L^8(I, L^4)} \right)\\
& \lesssim \; \| \theta^\eps(t) \|_{L^2} + \| L^\eps \|_{L^1(I, L^2)} + \eps^{1/8}\left(\tau^{3/4} + \eps^{5/4}e^{Ct} \right) \\
& \; \; \times \left( \eps^{-1/8} \| \theta^\eps\|_{L^1(I,L^2)} + \eps^{- 1/8} \| L^\eps \|_{L^1(I, L^2)} +\eps^{3/4}e^{Ct} \right),\\
& \lesssim \; \| \theta^\eps(t) \|_{L^2} + \sqrt{\eps}e^{Ct}+ \left(1+\eps^{7/4}e^{Ct} \right)\| \theta^\eps\|_{L^1(I,L^2)},
\end{align*}
where we have used the previous estimates about each term, and where we have used $\tau \leq 1$ since it is fixed small. 
Then, for $t \geq0$, we can write $\displaystyle{[0,t] \subset \bigcup_{j=0}^{N} [j\tau, (j+1)\tau]}$, for some integer $N$; and we obtain
\begin{equation*}
 \| \theta^\eps \|_{L^\infty([0,t], L^2)} \lesssim \| \theta^\eps(0) \|_{L^2} + 
\sqrt{\eps}e^{Ct}+ \left(1+\eps^{7/4}e^{Ct} \right)\| \theta^\eps\|_{L^1([0,t],L^2)}.
\end{equation*}
Since $t\leq t^\eps$, we have $\eps^{7/4}e^{Ct} < 1$ which gives
\begin{equation*}
 \| \theta^\eps \|_{L^\infty([0,t], L^2)}  \lesssim \; \| \theta^\eps(0) \|_{L^2} + 
\sqrt{\eps}e^{Ct}+ \| \theta^\eps\|_{L^1([0,t],L^2)},
\end{equation*}
Applying Gronwall Lemma on $\|\theta^\eps (t)\|_{L^2}$, we deduce :
\begin{equation}
\label{theta2}
 \forall t \in[0,t^\eps], \quad \|\theta^\eps (t)\|_{L^2} \leq C_0 \sqrt{\eps} e^{C_1t},
\end{equation}
where the constants are independent of $\eps$. 
Besides, combining \eqref{theta1} and \eqref{theta2}, we obtain 
\begin{equation}
 \label{theta3}
 \forall t \in[0,t^\eps], \quad \| \theta^\eps \|_{L^8(I, L^4)}  \lesssim \eps^{3/8}e^{Ct}.
\end{equation}
The proof is then completed by checking how long the boostrap assumption \eqref{bootstrap bis} holds. We differentiate the equation satisfied 
by $\theta^\eps$ and arguing as before, with Strichartz estimates we obtain :
\begin{align*}
\| \eps \nabla \theta^\eps \|_{L^8(I, L^4)}  \lesssim & \; \eps^{-1/8} \| \eps \nabla \theta^\eps (t)\|_{L^2} 
+ \eps^{- 1/8} \| \nabla V \theta^\eps \|_{L^1(I, L^2)} \\
&\quad  + \eps^{- 1/8} \| \eps \nabla L^\eps \|_{L^1(I, L^2)} 
+ \eps^{- 1/4}\| \eps \nabla NL^\eps \|_{L^{8/7}(I, L^{4/3})}.
\end{align*}
We recall the following estimate on the nonlinearity, obtained thanks to H\"older inequality :
\begin{align*}
 \eps^{-1/4}\| \eps \nabla NL^\eps \|_{L^{8/7}(I, L^{4/3})} \\
\lesssim & \; \eps^{1/4} 
\left( \|\varphi^\eps \|_{L^{8/3}(I,L^4)}^2 + \| \theta^\eps \|_{L^{8/3}(I,L^4)}^2 + \eps^2 \|g^\eps \|_{L^{8/3}(I,L^4)}^2 \right)\\
&  \left( \| \eps \nabla \theta^\eps\|_{L^8(I, L^4)} + \eps^2 \| \nabla (g^\eps \chi_-)\|_{L^8(I, L^4)} \right) \\
&  + \eps^{1/4}\| \varphi^\eps \; \eps\nabla \varphi^\eps\|_{L^{8/6}(I,L^2)}
\left(\| \theta^\eps\|_{L^8(I, L^4)} + \eps \|g^\eps\|_{L^8(I, L^4)} \right) \\
 \lesssim & \; \tau^{3/4} (1+\eps^{7/4}e^{Ct})\| \eps \nabla \theta^\eps\|_{L^8(I, L^4)}  + \tau^{3/4} \eps^{3/8} e^{Ct},
\end{align*}
where we have used the exponential control \eqref{B}, which is true in the case $d=1$, and \eqref{theta3}. 
For $t\leq t^\eps$, arguing as before, by fixing 
$\tau$ very small to absorb the nonlinearity, we obtain
\begin{equation*}
 \| \eps \nabla \theta^\eps \|_{L^8(I, L^4)}  \lesssim \eps^{-1/8} \| \eps \nabla \theta^\eps\|_{L^1(I,L^2)}
 + \eps^{- 1/8} \| \eps \nabla L^\eps \|_{L^1(I, L^2)} + e^{Ct}\eps^{3/8}.
\end{equation*}
Using Strichartz estimates again on $I$, we find
\begin{align*}
 \| \eps \nabla \theta^\eps \|_{L^\infty(I, L^2)} \lesssim &  \; \| \eps \nabla \theta^\eps (t)\|_{L^2} 
+ \| \nabla V \theta^\eps \|_{L^1(I, L^2)}+ \| \eps \nabla L^\eps \|_{L^1(I, L^2)} \\
& \quad + \eps^{- 1/8}\| \eps \nabla NL^\eps \|_{L^{8/7}(I, L^{4/3})}, \\
 \lesssim & \; \eps^{1/8} \left(\tau^{3/4} (1+\eps^{7/4}e^{Ct})\| \eps \nabla \theta^\eps\|_{L^8(I, L^4)}
+ e^{Ct}\eps^{3/8} \right) \\
& \quad + \| \eps \nabla \theta^\eps (t)\|_{L^2} + \eps^{1/2}e^{Ct} \\
 \lesssim & \| \eps \nabla \theta^\eps (t)\|_{L^2} +\tau^{3/4} (1+\eps^{7/4}e^{Ct})\| \eps \nabla \theta^\eps\|_{L^1(I,L^2)}  \\
& \quad  + \| \eps \nabla L^\eps \|_{L^1(I, L^2)} + e^{Ct}\eps^{7/4} + e^{Ct}\eps^{1/2} \\
 \lesssim & \| \eps \nabla \theta^\eps (t)\|_{L^2} + \eps^{1/2}e^{Ct}+ (1+\eps^{7/4}e^{Ct})\| \eps \nabla \theta^\eps\|_{L^1(I,L^2)}
\end{align*}
We recover $[0,t]$ with a finite number of intervals of the form $[j\tau,(j+1)\tau]$ and obtain
\begin{align*}
\| \eps \nabla \theta^\eps \|_{L^\infty([0,t], L^2)} & \lesssim \;  \eps^{1/2}e^{Ct}
+ (1+\eps^{7/4}e^{Ct})\| \eps \nabla \theta^\eps\|_{L^1(I,L^2)} \\
& \lesssim \; \eps^{1/2}e^{Ct} + \| \eps \nabla \theta^\eps\|_{L^1(I,L^2)},                                                                               
\end{align*}
with $t\leq t^\eps$ such that $\eps^{7/4}e^{Ct} < 1$.
Thanks to Gronwall lemma, we finally find  :
\begin{equation}
\label{theta4}
 \forall t \in [0,t^\eps], \quad \| \eps \nabla \theta^\eps \|_{L^2} \leq C'_0 \sqrt{\eps} e^{C'_1t}.
\end{equation}
Thanks to Gagliardo-Nirenberg inequality, we now write
\begin{align*}
 \| \theta^\eps (t)\|_{L^4}\lesssim & \; \eps^{-1/4} \| \theta^\eps (t)\|_{L^2}^{3/4}\|\eps \nabla \theta^\eps (t)\|_{L^2}^{1/4} \\
\lesssim & \; \eps^{-1/4}\sqrt{\eps} e^{Ct},
\end{align*}
where we have used \eqref{theta2} and \eqref{theta4}. We infer that the bootstrap argument \eqref{bootstrap bis} holds, at least, when 
$t\leq c t^\eps$ since we have
$$\forall t \in [0,t^\eps], \quad \eps^{-1/4}\sqrt{\eps} e^{Ct} \ll \eps^{-1/8},$$
with $t^\eps \leq c \log (\eps^{-1})$ for a suitable $c$, and this concludes the proof.
\section{Nonlinear superposition}
\label{NLsuperposition}
\subsection{General considerations}
Theorems \ref{superposition1} and \ref{superposition2} follow by the same meth\-ods as in Theorem \ref{main}. 
The main difference comes from the nonlinearity: nonlinear interaction terms appear. In this section, we will give the method for a 
nonlinear superposition of two data polarized along different modes. 
The procedure applied in this subsection is exactly the same if we consider same 
eigenspaces. We set
$$ w^\eps = \psi^\eps- \varphi^\eps_+\chi_+ - \varphi^\eps_-\chi_- +\eps g^\eps, $$
where $g^\eps$ is the sum of two correction terms, similar to the one defined in Section~\ref{strategy}:
$$g^\eps = g^\eps_+ \chi_+ + g^\eps_- \chi_-,$$ 
where the function $g^\eps_+$ solves the scalar Schr\"odinger equation
\begin{equation*}
i \eps \d_t g^\eps_+ + \dfrac{\eps^2}{2} \Delta g^\eps_+ - \lambda_+(x) g^\eps_+ = \varphi^\eps_- r_+ \quad ; \quad g^\eps_+(0,x)=0;
\end{equation*}
and the function $g^\eps_-$ solves
\begin{equation*}
i \eps \d_t g^\eps_- + \dfrac{\eps^2}{2} \Delta g^\eps_- - \lambda_-(x) g^\eps_- = \varphi^\eps_+ r_- \quad ; \quad g^\eps_-(0,x)=0,
\end{equation*}
where 
$$r_+ (t,x) = -i \left\langle d\chi_+(x)\xi^+(t), \chi_-(x) \right\rangle \quad ; 
\quad r_- (t,x) = -i \left\langle d\chi_-(x)\xi^-(t), \chi_+(x) \right\rangle.$$
The function $w^\eps(t)$ then solves
\begin{equation*}
 i \eps \d_t w^\eps + \dfrac{\eps^2}{2} \Delta w^\eps - V(x) w^\eps = \eps NL^\eps + \eps L^\eps \quad ; \quad w^\eps(0,x) = 0,
\end{equation*}
with
\begin{multline*}
 NL^\eps = \eps^{d/2} \left( | w^\eps + \varphi^\eps_+ \chi_+ + \varphi^\eps_- \chi_- + \eps g^\eps|^2 
\left( w^\eps + \varphi^\eps_+ \chi_+ + \varphi^\eps_- \chi_- + \eps g^\eps \right)\right. \\
 \left. - |\varphi^\eps_+|^2 \varphi^\eps_+ \chi_+ - |\varphi^\eps_-|^2 \varphi^\eps_- \chi_- \right).
\end{multline*}
and
$$L^\eps = \mathcal{O} \left( \sqrt{\eps}e^{Ct} \right) + \left[\dfrac{\eps^2}{2} \Delta, \chi_+ \right]g^\eps_+ +
 \left[\dfrac{\eps^2}{2} \Delta, \chi_- \right]g^\eps_-   = \mathcal{O} \left( \sqrt{\eps}e^{Ct} \right),$$
where the estimate holds in $H_\eps^1$, using Proposition \ref{L2}.
To deal with the nonlinearity, we add and subtract the term 
$$\eps^{d/2} |\varphi^\eps_+ \chi_+ + \varphi^\eps_- \chi_-|^2 \left( \varphi^\eps_+ \chi_+ + \varphi^\eps_- \chi_-\right),$$
and obtain
$$NL^\eps = N_I^\eps + N_S^\eps,$$
where we have 
$$N_I^\eps = |\varphi_+^\eps|^2 \varphi^\eps_- \chi_- + |\varphi^\eps_-|^2 \varphi^\eps_+ \chi_+, $$
and the following pointwise estimates
\begin{align*}
 |N_I^\eps| &\lesssim \eps^{d/2} \left( |\varphi^\eps_+|^2 |\varphi^\eps_-| + |\varphi^\eps_-|^2 |\varphi^\eps_+|\right), \\
 |N_S^\eps| &\lesssim \eps^{d/2} \left( |\varphi^\eps_+|^2 + |\varphi^\eps_-|^2 + |w^\eps|^2 + \eps^2 |g^\eps|^2 \right) 
\left( |w^\eps| + \eps |g^\eps| \right).   
\end{align*}
The procedure to estimate the term $N_S^\eps$ is exactly the same used to deal with the nonlinearity in Section \ref{last}.
The only point remaining concerns the analysis of $\int_0^t \|N_I^\eps (s)\|_{H^1_\eps} ds$.
We have
\begin{equation*}
 \eps^{d/2} \| (\varphi^\eps_+)^2 \varphi^\eps_-\|_{L^2(\R^d)} = 
\left\|\left( u_+\left( t,y- \dfrac{x^+(t)-x^-(t)}{\sqrt{\eps}}\right)\right)^2  u_-(t,y) \right\|_{L^2(\R^d_y)}
\end{equation*}
and the term $|\varphi^\eps_-|^2 |\varphi^\eps_+|$ is handled in the same way, as their contribution play the same role.
We leave out the other terms which are needed in view of a $H^1_\eps$ estimate, since they create no trouble.
\\
\noindent 
The estimation of $N_I^\eps$ is given by the following lemma. 
The proof is based on the strategy of \cite{Ca-p} with adaptation required by the fact that we are in the case $d>1$.
\begin{lemma}
 \label{Mk}
Let $T >0$, $0<\gamma<1/2$ and
\begin{equation}
  \label{eq:ieps}
  I^\eps(T)=\{t\in[0,T],\;\;|x^+(t)-x^-(t)|\le \eps^\gamma\}.
\end{equation}
Then, for all integer $k$, $k > d/2$, there exists a constant $C=C(k)$ such that
$$\int_0^T \|N_I^\eps(t) \|_{H^1_\eps} dt \lesssim (M_{k+2}(T))^3 \left( T \eps^{k(1/2-\gamma)}+ |I^\eps(T)| \right),$$
where $M_k(T) = \max \left(M_k^+(T), M_k^-(T) \right)$, with
$$M_k^\pm (T) = \sup \left\lbrace \|\left\langle x \right\rangle^\alpha \d_x^\beta u_\pm \|_{(L^\infty[0,T];L^2(\R^d))}; 
\quad |\alpha|+|\beta|\leq k \right\rbrace.$$
\end{lemma}
Our next objective is to evaluate the quantity $|I^\eps(T)|$, for $T>0$. 
\\
\noindent
It has to be noticed that the arguments of \cite{Ca-p} which allow us to deal with large times cannot be generalized to 
higher dimension : they are specific to the one-dimensional case.

\subsection{Nonlinear superposition for data belonging to different modes}
\label{different}
\begin{lemma}
Let $T>0$ and 
$$ \Gamma = \inf_{x \in \R^d} | E^+ - E^- - (\lambda_+(x) - \lambda_- (x))|, $$
and suppose $\Gamma >0$. Then, for $0 < \gamma < 1/2$, 
for $T >0$, independent of $\eps$, we have :
$$ | I^\eps (T) | \lesssim \dfrac{\eps^\gamma}{\Gamma^2},  $$
where $I^\eps (T)$ is defined in \ref{eq:ieps}
\end{lemma}

\begin{proof}
We consider $J^\eps(T)$ a maximal interval, included in $I^\eps (T)$ and $N^\eps(T)$ 
the number of such intervals. 
We have the following estimate :
\begin{equation}
\label{Neps}
  |I^\eps(T)| \leq |J^\eps (T)| \times N^\eps(T).
\end{equation}
Let $z$ be defined by $z(t) = |x^+(t) - x^-(t)|^2$. We first prove that
\begin{equation}
\label{ddot}
 \exists \eps_0>0, \; \exists 0<\delta<1, \; \forall \eps \in ]0,\eps_0], \;  \forall t \in I^\eps (T), \; 
\textrm{we have} \; \ddot{z}(t) \geq \delta \Gamma^2 >0.
\end{equation}
\textbf{Step zero : Proof of \eqref{ddot}.}
\\
\noindent
We have for $t \in J^\eps(T)$ : 
\begin{align}
\notag
& \ddot{z}(t)  =
 2|\xi^+ (t) - \xi^- (t)|^2 - 2(x^+(t)-x^-(t)).\left(\nabla \lambda_+ (x^+(t))- \nabla \lambda_- (x^-(t))\right),\\
\label{ddot1}
& \ddot{z}(t)  \geq 2 |\xi^+ (t) - \xi^- (t)|^2 - C \eps^\gamma.
\end{align}
Besides, we have :
\begin{align*}
 | E^+ - E^- - (\lambda_+ (x^+(t)) - \lambda_- (x^-(t)))| & = 
| E^+ - E^- - (\lambda_+ (x^+(t)) - \lambda_- (x^+(t))) \\
& \quad + \lambda_- (x^+(t))- \lambda_- (x^-(t))| \\
& \geq | E^+ - E^- - (\lambda_+ (x^+(t)) - \lambda_- (x^+(t))| \\
& \quad - |\lambda_- (x^+(t))- \lambda_- (x^-(t))| \\
& \geq \Gamma - C\eps^\gamma,
\end{align*}
since $|\lambda_- (x^+(t))- \lambda_- (x^-(t))| \leq C |x^+(t)- x^-(t)| $. If $C\eps^\gamma \ll 1$, we obtain
\begin{equation}
\label{energies}
| E^+ - E^- - (\lambda_+ (x^+(t)) - \lambda_- (x^-(t)))| \geq \dfrac{\Gamma}{2}.
\end{equation}
Then we write, by the definitions of the energies : 
\begin{equation*} 
| E^+-E^--(\lambda_+(x^+(t))-\lambda_-(x^-(t)))|
 = \dfrac{1}{2} \left\lvert (\xi^+ (t)-\xi^-(t)).(\xi^+(t)+\xi^-(t))\right\rvert,
\end{equation*}
whence by \eqref{trajestimate}:
\begin{equation*}
| E^+-E^--(\lambda_+(x^+(t))-\lambda_-(x^-(t)))|\lesssim |\xi^+ (t)-\xi^-(t)|,
\end{equation*}
and we obtain by \eqref{ddot1} and \eqref{energies}, for $C\eps^\gamma$ such that $C'\Gamma^2-C\eps^\gamma \geq \Gamma^2/2$ : 
\begin{equation*}
 \delta\Gamma^2 \leq \ddot{z}(t), \quad \textrm{with}\;0<\delta<1.
\end{equation*}
\textbf{Step one : Size of $J^\eps(T)$.}
\\
\noindent
Let us now consider $\tau,\tau' \in J^\eps(T)$ and find a lower bound of $|\tau-\tau'|$.
\\
\noindent 
The derivatives of $z$ are given by :
\begin{align*}
\dot{z}(t) & = 2 (x^+(t) - x^-(t)).(\xi^+(t) - \xi^-(t)), \\
\ddot{z}(t) & = 2 |\xi^+ (t) - \xi^- (t)|^2 - 2 (x^+(t)-x^-(t)).\left( \nabla \lambda_+ (x^+(t)) - \nabla \lambda_- (x^-(t))\right) 
\end{align*}
There exists $t^* \in \left] \tau, \tau' \right[$ such that 
\begin{equation}
\label{*}
 |\dot{z}(\tau') - \dot{z}(\tau) | = |\tau' - \tau |\; \ddot{z}(t^*).
\end{equation}
On one hand, we have by \eqref{trajestimate}
\begin{align}
\notag
 |\dot{z}(\tau') - \dot{z}(\tau) | & \leq \; 2 | x^+(\tau') - x^- (\tau')|\;|\xi^+(\tau') - \xi^- (\tau')| \\
\notag
& \quad + 2 | x^+(\tau) - x^- (\tau)|\;|\xi^+(\tau) - \xi^- (\tau)|, \\
\label{**}
& \lesssim  \; |\dot{z}(\tau') - \dot{z}(\tau) | \; \eps^\gamma.
\end{align}
On the other hand, we have \eqref{ddot} for $t^* \in J^\eps(T)$.
Therefore, in view of \eqref{*},\eqref{**} and \eqref{ddot} we infer
$$\Gamma^2 |\tau' - \tau| \lesssim |\tau' - \tau| \ddot{z}(t^*) = |\dot{z}(\tau') - \dot{z}(\tau) | \lesssim \eps^\gamma,$$
whence
\begin{equation*}
 |\tau - \tau' | \lesssim \dfrac{\eps^\gamma}{\Gamma^2}, \quad \textrm{and} \quad |J^\eps(T)| \lesssim \dfrac{\eps^\gamma}{\Gamma^2}.
\end{equation*}
\textbf{Step two : Estimation of $N^\eps(T)$.}
\\
\noindent
The difficulty is to prove that the number of interval $J^\eps(T)$ contained in $I^\eps(T)$ is independent of $\eps$. 
For this reason, we first consider a fixed $\eps$.
\\
\noindent
Set $\eps_0 >0$ a fixed constant, small enough to have :
$$C'\Gamma^2 - C\eps_0 \geq \Gamma^2/2.$$
We consider $I^{\eps_0}(T)$, let us prove that there is a finite number of intervals in this fixed set.
We argue by contradiction, assuming that there is an infinite number of intervals such that $z(t) \leq \eps_0^{2\gamma}$.
\\
\noindent
We consider a sequence $(s_n)_{n\in \N}$, included in $[0,T]$ and such that each term is in a 
connected interval of $I^{\eps_0}(T)$, which does not contain an other term of the sequence. We can assume that the sequence is monotonic 
(let us say strictly increasing to fix ideas). 
By compactness of $[0,T]$, a subsequence of $(s_n)_n$ converges to some $s \in [0,T]$, with 
$$z(s) \leq \eps_0^{2\gamma}.$$
Besides, there exists $(t_n)_{n\in \N}$, such that for all 
$$n \in \N,\;  t_n \in ]s_n, s_{n+1}[ \quad \textrm{with}\quad z(t_n) > \eps_0^{2\gamma}$$
and $$\dfrac{d}{dt} z(t)_{|_{t=t_n}}=0,$$ for all integer $n$. 
Using the same argument of compactness, we infer that this sequence $(t_n)$ converges to $s$, 
with $z(t) \geq \eps_0^{2\gamma}$, we deduce that 
$$z(t) = \eps_0^{2\gamma}, \quad \textrm{and} \quad \dfrac{d}{dt} z(t){|_{t=s}}=0.$$
We have $$\dfrac{d}{dt} z(t){|_{t=t_n}}= \dfrac{d}{dt} z(t){|_{t=t_{n+1}}}= 0,$$ 
for all $n \in \N$; thanks to Rolle's theorem, 
there exists en sequence $(r_n)_{n\in \N}$ such that for all $n \in \N$, $r_n \in ]t_n, t_{n+1}[$ and
$$\dfrac{d^2}{dt^2} z(t){|_{t=r_{n}}}= 0.$$ 
Arguing as before, we infer that $(r_n)$ converges to $s$, with $\dfrac{d^2}{dt^2} z(t){|_{t=s}}= 0$. 
But \eqref{ddot} for $s$ give $\ddot{z} (s)>0$. 
Hence a contradiction with \eqref{ddot}.
\\
\noindent
For $\eps$ sufficiently small, such that $\eps \leq \eps_0$, $J^\eps(T)$ is included in $I^{\eps_0}(T)$. 
Besides $z(t)\geq 0$ and $\ddot{z}(t) \geq \delta \Gamma^2 >0$, which implies that $z$ is a positive and strictly convex function. 
We infer that in each interval of $I^{\eps_0}(T)$, there is exactly one interval where $z$ is very small, 
such that $z(t) \leq \eps^{2 \gamma}$. 
This implies that the number of such intervals of $J^\eps(T)$ is the same as in $I^{\eps_0}(T)$; 
$N^\eps(T)$ is bounded by a constant independent of $\eps$, and this concludes the proof.
\end{proof}

\subsection{Nonlinear superposition of data belonging to same modes}
Let us first notice that $I^\eps(T)$ has to be rewritten :
Let $T >0$, $0<\gamma<1/2$. We set :
\begin{equation*}
I^\eps(T)=\{t\in[0,T],\;\;|x_1(t)-x_2(t)|\le \eps^\gamma\}.
\end{equation*}
Then, to estimate the size of $I^\eps(T)$, we have :
\label{same}
\begin{lemma}
Let $T>0$, independent of $\eps$. Then, there exists $C>0$, such that 
$$|I^\eps(T)| \leq C \eps^\gamma.$$
\end{lemma}
\begin{proof}[Sketch of the proof]
The proof is based on Lemma 6.2 of \cite{CF-p}. In this case, we consider classical trajectories built with the same eigenvalue, 
which is similar to the scalar case, with a scalar potential. It has to be noticed that for 
$$(x_1(0),\xi_1(0)) \neq (x_2(0),\xi_2(0)),$$
we have 
$$(x_1(t),\xi_1(t)) \neq (x_2(t),\xi_2(t))$$
for all times $t$, since the trajectories solve the same ODE system. Therefore, on $[0,T]$, the curves $x_1(t)$ and $x_2(t)$ cross on a 
finite number of isolated points, where $\xi_1(t)\neq \xi_2(t)$. Then, the control of the quantity $|\xi_1(t)- \xi_2(t)|$ in \eqref{ddot} 
follows without any assumption (See \cite{CF-p} for details).
\end{proof}

\begin{remark}
To complete the proof of Propositions \ref{superposition1} and \ref{superposition2}, it remains to perform a bootstrap argument, 
similar to the one in the proof of Theorem \ref{main}, (see \cite{CF-p} for details) which gives a finer condition on $\gamma$:
$$\eps^{\gamma - d/4} \ll \eps^{-d/8}, $$
and it is equivalent to $\gamma > d/8$ which is compatible with $\gamma < 1/2$. 
\end{remark}
\smallbreak
In both situations, for large time case, we cannot use the same method as in~\cite{CF-p} for a scalar potential, 
or in \cite{Adiab}, in the one-dimensional case, to find the number of maximal intervals.
\begin{remark}
\label{RmqNeps}
However, assuming \eqref{B} is satisfied; and
$$ | N^\eps(t)| \lesssim e^{Ct},$$
where $N^\eps(t)$ is defined in \eqref{Neps}, then one can prove the following result :
there exists $C>0$ independent of $\eps$ such that
$$\sup_{t \leq C \log \log (\eps^{-1})} \| w^\eps (t)\|_{H_\eps^1} \Tend\eps 0 0.$$ 
\end{remark}
\begin{remark}
Let us notice that if the approximation of Theorem \ref{main} is valid up to a time $t= C \log (\eps^{-1})$, 
then, Theorems \ref{superposition1} and \ref{superposition2} will be also valid up to an analogue time.
\end{remark}

\appendix
\section{Strichartz estimates}
\label{StrichartzAnnexe}
In view of Remark 4 of \cite{FerRou}, Proposition 3 of \cite{FerRou} writes:
\begin{proposition}
\label{A1}
Consider $T>0$ and $(p,q)$ an admissible pair. Then, there exists a constant $C=C(q)$ such that
\begin{equation*}
 \left\|e^{i\frac{t}{\eps}P(\eps)}u_0^\eps \right\|_{L^p([0,T],L^q(\R^d))} \leq C \eps^{-1/p} \|u_0^\eps\|_{L^2(\R^d)},
\end{equation*}
\end{proposition}
\noindent
and Corollary 1 writes:
\begin{corollary}
Consider $T>0$ and $(p,q)$ an admissible pair. Then, there exists a constant $C=C(q)$ such that
\begin{equation*}
 \left\|\int_0^t e^{i\frac{t-s}{\eps}P(\eps)} f^\eps(s) ds\right\|_{L^p([0,T],L^q(\R^d))} \leq 
C \eps^{-1/p} \|f^\eps\|_{L^1([0,T],L^2(\R^d))}.
\end{equation*}
\end{corollary}
Note that these results of \cite{FerRou} crucially use the long range property of $V$, which allows to prove 
resolvent estimates. Proposition \ref{A1} gives the first Strichartz estimate \eqref{Strichartz1}. 
Let us prove that we also have \eqref{Strichartz2}, namely:
\begin{equation*}
\left\|\int_{0}^{t} e^{i \frac{t-s}{\eps}P(\eps)}f^\eps(s) ds \right\|_{L^{p_1} ([0,T], L^{q_1}(\R^d))} 
\leq C \eps^{-1/p_1 - 1/p_2} \|f^\eps \|_{L^{p'_2} ([0,T], L^{q'_2}(\R^d))},
\end{equation*}
where $(p_1,q_1)$ and $(p_2,q_2)$ are admissible pairs, and $C$ independent of $\eps$.
\\
\noindent
By Christ-Kiselev's theorem, it is sufficient to estimate the $L^{p_1}([0,T],L^{q_1})-$norm of
$$h^\eps(t) :=\int_0^T e^{i\frac{t-\tau}{\eps}P(\eps)} f^\eps(\tau) d\tau.$$
Using \eqref{Strichartz1}, we obtain
\begin{align*}
\| h^\eps\|_{L^{p_1}([0,T],L^{q_1})} & \leq C(q_1) \eps^{-1/p_1}\|h^\eps(0) \|_{L^2}, \\
& \leq C(q_1) \eps^{-1/p_1} \left\|\int_0^T e^{-i\frac{s}{\eps}P(\eps)} f^\eps(s) ds \right\|_{L^2}.
\end{align*}
The dual inequality of \eqref{Strichartz1} for $(p_2,q_2)$ admissible pair gives
$$\left\|\int_0^T e^{-i\frac{\tau}{\eps}P(\eps)} f^\eps(\tau) d\tau \right\|_{L^2} \leq C(q_2) \eps^{-1/p_2}
\|f^\eps \|_{L^{p'_2}([0,T],L^{q'_2})}. $$
Combining these estimates, we finally obtain \eqref{Strichartz2}.

\section{Global existence of the exact solution}
In view of \cite{Ca-p}
Section \ref{StrichartzAnnexe} allows us to prove global existence and uniqueness of the solution of \eqref{NLS0}-\eqref{data}, 
for fixed $\eps>0$:
\begin{proposition}
 \label{existence}
If $V$ satisfies Assumption \ref{assumption}, and $\psi^\eps_0 \in L^2(\R^d)$, there exists a unique, global, solution to 
\eqref{NLS0}-\eqref{data}
$$ \psi^\eps \in \mathcal{C} \left( \R, L^2(\R^d)\right)  \cap L_{loc}^{8/d}\left( \R, L^4(\R^d)\right).$$
Moreover, the $L^2-$norm of $\psi^\eps$ does not depend on time
$$\| \psi^\eps(t)\|_{L^2(\R^d)} = \| \psi_0^\eps\|_{L^2(\R^d)}, \quad \forall t \in \R. $$
We denote by $E(t)$ the energy, given by :
$$E(t) = \dfrac{1}{2} \| \nabla u(t)\|_{L^2}^2 + \dfrac{\Lambda}{4} \| u(t) \|_{L^4}^4 + \dfrac{1}{2} \int_{\R^d} 
\left\langle Q^+(t)x;x \right\rangle |u(t,x)|^2 dx.$$
This quantity does not depend on time : $E(t) = E(0), \quad \forall t \in \R$.
\end{proposition}
\begin{proof}[Sketch of the proof]
From the above-mentionned results, it follows that local in time Strichartz estimates are available. 
Therefore, using a fixed point argument, one can prove local existence of the solution. 
Then, using the conservation of the $L^2-$norm and of the energy, one can infer that the solution is global. 
See \cite[Remark 5]{FerRou} and \cite{Glassey} for the details.
\end{proof}
\noindent
For fixed $\eps$, it is actually possible to prove global existence of the solution under weaker assumptions : 
$V$ has to be \textit{at most quadratic} :
assuming $\rho$ and $\rho_0$ are at most quadratic and $\om$ is bounded with bounded derivatives, 
the author of \cite{Ca-p} obtains global existence of 
the solution.
\\[3mm]
\textbf{Acknowledgments.} The author wishes to thank Thomas Duyckaerts and Clotilde Fermanian 
for suggesting this problem and for many helpful advices during the preparation of the paper. 
The author also acknowledges financial support by the grant \textit{ERC DISPEQ}.

\bibliographystyle{amsplain}
\bibliography{biblio}
\end{document}